\documentclass[a4paper,reqno]{amsart}

\usepackage[utf8]{inputenc}
\usepackage{graphicx}
\usepackage{amsmath}
\usepackage{amsfonts}
\usepackage{amssymb}
\usepackage{amsthm}
\usepackage{pgf,tikz}
\usetikzlibrary{calc, positioning}
\usetikzlibrary{arrows}
\usepackage{nicefrac}
\usepackage{cleveref}
\usepackage{mathtools}
\usepackage{microtype}
\usepackage{fullpage}
\usepackage[foot]{amsaddr}

\numberwithin{equation}{section}

\theoremstyle{definition}
\newtheorem{defi}{Definition}[section]

\theoremstyle{plain}
\newtheorem{theorem}[defi]{Theorem}
\newtheorem{proposition}[defi]{Proposition}
\newtheorem{lemma}[defi]{Lemma}
\newtheorem{corollary}[defi]{Corollary}

\theoremstyle{remark}

\usepackage{graphicx}
\usepackage{caption}
\usepackage[labelformat=simple]{subcaption}

\newcommand{\dx}{\, \mathrm{d}}
\renewcommand{\tilde}{\widetilde}
\renewcommand{\P}{\mathbb{P}}
\newcommand{\R}{\mathbb{R}}
\newcommand{\N}{\mathbb{N}}
\newcommand{\mbf}[1]{\mathbf{#1}}
\newcommand{\Sinf}{\Sigma^{\infty}}
\newcommand{\M}{\mathcal{M}}
\newcommand{\aq}[1]{\sim_{#1}}
\newcommand{\A}[2]{A_{#1}^{(#2)}}
\newcommand{\qur}[2]{#1/_{\aq #2}}
\newcommand{\bsr}{\boldsymbol{\rho}}

\def\tri
{
  \ifnum\count5=\topDepth
    \draw 
      (\the\numexpr\count7+2\relax,\the\count8) -| 
      (\the\count7,\the\numexpr\count8+2\relax) -- 
      cycle;
  \else
    \begingroup
      \advance\count5 by 1
      \multiply\count6 by 2
        \tri
      \begingroup
        \advance\count7 by\count6
        \tri
      \endgroup
      \begingroup
        \advance\count8 by\count6
        \tri
      \endgroup
    \endgroup
  \fi
}

\pgfmathsetmacro{\mySqrt}{sqrt(.75)}

\begin{document}
	
\title{The Sierpi\'nski gasket as the Martin boundary of a non-isotropic Markov chain}

\author[M.\ Kesseb\"ohmer]{M.\ Kesseb\"ohmer}
\address[M.\ Kesseb\"ohmer \& K.~Sender]{FB~3--Mathematik und Informatik, Universit\"at Bremen, 28359 Bremen, Germany}
\author[T.\ Samuel]{T.\ Samuel}
\address[T.\ Samuel]{\parbox[t]{0.875\textwidth}{Mathematics Department, California Polytechnic State University, San Luis Obispo, CA, USA and Institut Mittag-Leffler, Aurav\"agen 17, Djursholm, Sweden\\[-0.75em]}}
\author{K.~Sender}
	


\maketitle
	
\begin{abstract}
In 2012 Lau and Ngai, motivated by the work of Denker and Sato, gave an example of an isotropic Markov chain on the set of finite words over a three letter alphabet, whose Martin boundary is homeomorphic to the Sierpi\'nski gasket.  Here, we extend the results of Lau and Ngai to a class of non-isotropic Markov chains.  We determine the Martin boundary and show that the minimal Martin boundary is a proper subset of the Martin boundary. In addition, we give a description of the set of harmonic functions. 
\end{abstract}

\section{Introduction}
	
The concept of Martin boundaries for Markov chains stems from the work of Martin, Doob and Hunt and has close ties to harmonic analysis and potential theory.  Indeed, in solving the Dirichlet problem for arbitrary domains in $ \R^n $, Martin introduced the notion of an ideal boundary \cite{MR0003919}.  Roughly 20 years later Doob \cite{MR0107098} and Hunt \cite{MR0123364} gave a probabilistic version for Markov chains, which is now known as a Martin boundary. This provides a motivation for constructing Markov chains with fractal Martin boundaries, as it offers a probabilistic approach to the study of analysis on fractals, which has recently attracted much attention -- see for example \cite{MR3042410,MR2746525,MR2912440,MR1076617,Kig,MR3203408,MR2095624,Str} and references therein.  We refer the reader to \cite{Dyn,KKS,MR1743100,MR2548569} for a general introduction to harmonic analysis and potential theory for Markov chains.
	
Denker and Sato \cite{MR1739300,MR1822915} created a Markov chain whose Martin boundary is homeomorphic to the Sierpi\'nski gasket (see Figure~\ref{fig:SG}), and used potential theory on the Martin boundary to induce a harmonic structure.  In \cite{DS2002} they identified a subclass of `strongly harmonic functions' on the Martin boundary which coincides with Kigami's canonical class of harmonic functions \cite{MR1076617,Kig,Str}.  Denker, Imai and Koch \cite{MR2283132} extended this construction to some \mbox{non-self-similar} Sierpi\'nski type gaskets and studied an associated Dirichlet form.  Further, there exists a family of metrics on the Martin boundary dependent on a family of scaling factors. In \cite{MR2180235} the Hausdorff, packing and information dimension of the Martin boundary with respect to this family of metrics was studied. The work of \cite{MR1739300,MR1822915} has been shown to encompass the pentagasket,~see~\cite{AI2002}.
	
The class of connected post critically finite \mbox{self-similar} sets, to which the Sierpi\'nski gasket belongs, has played a crucial role in the development of analysis on fractals, see for instance \cite{MR2746525,MR1076617,Kig,Str} and references therein. In \cite{MR2833578}, Ju, Lau and Wang built on the line of research initiated by Denker and Sato, by showing that, for a certain class of post critically finite \mbox{self-similar} sets, one may define a Markov chain whose Martin boundary is homeomorphic to the given set.  In all of the above considerations, the Markov chain is non-reversible and isotropic, where by isotropic we mean that the chain has equal probability to pass to the next state.
	
To our knowledge, the first representation of a connected post critically finite \mbox{self-similar} fractal set as the Martin boundary of an isotropic reversible Markov chain was given by Pearse \cite{2011arXiv11041650P}.  Lau and Wang showed in \cite{MR3649228} that for any contractive iterated function system, there is a naturally defined augmented tree, which is hyperbolic and whose hyperbolic boundary is H\"older equivalent to the \mbox{self-similar} set. Moreover, an iterated function system satisfies the open set condition if and only if this augmented tree has uniformly bounded degree. In \cite{KONG20171099}, Kong, Lau and Wong considered an isotropic reversible random walk on such an augmented tree and, using the results of Ancona \cite{A1988}, showed that the Gromov boundary, the Martin boundary, the minimal Martin boundary and the \mbox{self-similar} set are all homeomorphic. Further, under certain conditions, using an approach of Silverstein \cite{MR0350876}, they proved that the Martin kernel, which gives rise to the Martin metric and hence the Martin boundary, defines a non-local Dirichlet form. The work of Kong, Lau and Wong complements that of Series \cite{MR693661} who showed the following.  For a finitely generated non-elementary Fuchsian group $\Gamma$ without cusps, and a finitely supported probability measure $\mu$ on $\Gamma$, the Martin boundary of the random walk on $\Gamma$ with distribution $\mu$ is homeomorphic~to~the~limit~set~of~$\Gamma$.
	
In \cite{Lau}, Lau and Ngai defined an isotropic Markov chain on the set of finite words $ \Sigma^* \coloneqq \bigcup_{n \in \N_{0}} \Sigma^n $ over the alphabet $ \Sigma \coloneqq \{1,2,3\} $. They showed that the Martin boundary is homeomorphic to the Sierpi\'nski gasket, whereas unlike in the previous constructions, the minimal Martin boundary is a proper subset of the Martin boundary and coincides with the post critical set.  Additionally, they proved that the harmonic functions are precisely the canonical harmonic functions of Kigami.  This work has been extended to the Hata tree, a connected non-symmetric \mbox{self-similar} post critically finite fractal set, see \cite{Lau2}.  
	
Our contributions to this story and the purpose of this article is to extend the construction of  \cite{Lau} to the case when the Markov chain is non-isotropic.  Indeed, we consider a class of non-isotropic Markov chains dependent on a parameter $p \in (0, 1/2)$, and show that the Martin boundary and the minimal Martin boundary is independent of the choice of $p$.  We find this result interesting as the Martin boundary is defined via a metric, called the Martin metric, which is dependent on scaling factors and the parameter $ p $, see Section~\ref{sec_Martin_metric}.  Moreover, the theory of Ancona \cite{A1988} is not applicable in the setting of \cite{Lau}, and hence our setting, as the isoperimetric inequality is not satisfied.
	
The state space of our Markov chain will be the set of finite words $ \Sigma^*$. We regard each $ \Sigma^n $ as the set of vertices of a graph $ \Gamma^n $, which we will consider as level-$ n $ approximations of the Sierpi\'nski gasket, see Figures~\ref{fig:Graphs_Gamma_n} and~\ref{fig:P}. The Markov chain is defined as nearest neighbour random walk on each $ \Gamma^n $, except for three `boundary vertices'. When hitting one of these, the Markov chain moves to the next level, namely $ \Gamma^{n+1} $.  Our chain is constructed such that it stays with probability $ 2 p \in (0, 1) $ at the `outer part' of the graphs $ \Gamma^n $, and goes with probability $ q \coloneqq 1 - 2p $ to the `inner part' of $ \Gamma^n $.  We exclude $ p \in \{0, 1/2 \} $, since in this case the Martin metric is not a metric and hence, the Martin boundary is not well defined. We note that the Markov chain of \cite{Lau} occurs as a special case of our setting when $ p = 1/3$.   Our main results are Theorems \ref{thm_M_homeo_K}, \ref{thm_min_Martin_boundary} and \ref{thm:harmonic_functions}, where the key contribution to proving these results lies in Theorem~\ref{prop_limits}. 
	
	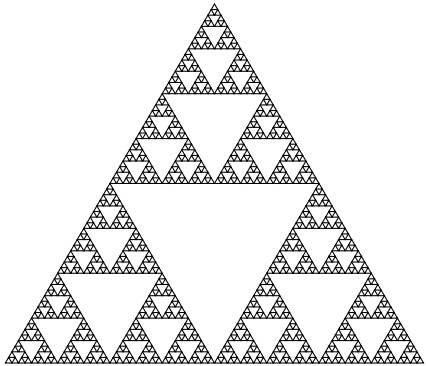
\begin{figure}
		\begin{center}
			\def\topDepth{6}
			
			\begin{tikzpicture}[scale=2.75*.5^\topDepth,cm={1,0,.5,\mySqrt,(0,0)}]
			\tri
			\end{tikzpicture}
			
			\caption{The Sierpi\'nski gasket.} \label{fig:SG}
		\end{center}
	\end{figure}
	
This article is structured as follows. In Section~\ref{basics} we give basic definitions and formally introduce the graphs $\Gamma_{n}$.  An important tool in identifying the Martin boundary with the Sierpi\'nski gasket will be what is referred to as the standard projection; this is matter of Section~\ref{Standard_projection}.  In Section~\ref{The_Markov_chain} we define our Markov chain $(X_{n})_{n \in \N}$ outlined above. Next, we give key hitting probabilities of $(X_{n})_{n \in \N}$ in Section~\ref{Random_matrix_product}. With this at hand, we may express the probability to move to the next level in the graph as a random matrix product. This is the main tool in \cite{Lau} and depends only on the underlying graph structure of the Markov chain. Here we observe that the framework of \cite{Lau} may be applied with some modifications. We investigate the limiting behaviour of the matrix product in Section~\ref{sec_limit_matrix_prod} and introduce the Martin metric in Section~\ref{sec_Martin_metric}.  Section~\ref{The_Green_function} is concerned with showing that the Green function and the Martin kernel can be extended to the set of infinite words over the alphabet $ \Sigma $.  In Section~\ref{The_Martin_metric_and_Martin_space}, we introduce the Martin boundary and describe how the homeomorphism of the Martin boundary and the Sierpi\'nski gasket is obtained. Section~\ref{sec_harm_fct} deals with determining the harmonic functions related to the Markov chain.  The non-trivial and challenging task of this work is to establish the limits of the sequences of hitting probabilities discussed in Section~\ref{Random_matrix_product}.  This is the focus of Section~\ref{sec_hitt_prob}.
	
	\section{Construction of the Markov chain} \label{Constr_Markov_chain}
	
	\subsection{Basic definitions}\label{basics}
	
	We write $\Sigma^n \coloneqq \{1,2,3\}^n $ for the \textsl{set of words of length $n \in \N_{0}$ over the alphabet $\Sigma \coloneqq \{1, 2, 3 \} $}, where following convention $\Sigma^0 \coloneqq \{ \vartheta \}$ is the set containing the \textsl{empty word} $\vartheta$. The \textsl{set of all finite words} is defined by $ \Sigma^* \coloneqq \bigcup_{n \in \N_0} \Sigma^n $ and the \textsl{set of all infinite words} by $\Sinf \coloneqq \{1,2,3\}^{\N}$. We let $\sigma \colon \Sinf \to \Sinf$ denote the \textsl{left shift map} which acts on infinite words as follows: $ \sigma(i_1 i_2 \ldots) = i_2 i_3 \ldots $ for $i_1 i_2 \ldots \in \Sinf $.
	
	For $a \in \Sigma $ and $n \in \N$, we write $a^{n}$ for the $n$-fold concatenation of $a$ with itself, and let $a^{\infty}$ be the infinite word with all letters equal to $a$.  For $\mbf{x} = \omega_1 \omega_2 \ldots \in \Sinf$ and $ n \in \N $, set $\mbf{x}|_n \coloneqq \omega_1 \omega_2 \ldots \omega_n \in \Sigma^n$.  We call $V^n \coloneqq \{ 1^n, 2^n, 3^n \} $ the \textsl{boundary of $\Sigma^n$} and call $\tilde{\Sigma}^n \coloneqq \Sigma^n \setminus V^n $ the \textsl{interior of $ \Sigma^n$}.  Similarly, we set $ V^{\infty} \coloneqq \left\{ 1^{\infty}, 2^{\infty}, 3^{\infty} \right\} $ and $\tilde{\Sigma}^{\infty} \coloneqq \Sinf \setminus V^{\infty} $.  
	
	For $ m,n \in \N $ with $ m \leq n $ and $ \omega \in \Sigma^{m - 1} $, the set $ \Delta_{\omega}^n \coloneqq \{ \omega i_m \cdots i_n \colon i_m, \ldots, i_n \in \Sigma \} $ is called a \textsl{$ (m, n) $-cell}.   We refer to a $ (n, n) $-cell as  a $ n $-cell.  An element of the set $\partial \Delta_{\omega}^{n} \coloneqq \{  \omega i^{n - m + 1} \colon i \in \Sigma \}$ is called an \textsl{outer vertex} of the $ (m, n) $-cell $ \Delta_{\omega}^n $. Notice a $ n $-cell consists only of outer vertices.
	
	If $ \Gamma = (V,U) $ is a graph with vertex set $ V $ and edges set $ U $, we let $ (x,y) \in U $ denote an undirected edge from $ x$ to $y$, where $x, y \in V $.  For $n \in \mathbb{N}$, we define the graph $\Gamma^{n}$ with vertex set $ \Sigma^n $ as follows.  Set $ U^1 \coloneqq \{ (1,2), (1,3), (2,3) \} $ and $ \Gamma^1 \coloneqq (\Sigma^1, U^1) $. 
	Assume that $ \Gamma^{n - 1} = (\Sigma^{n - 1}, U^{n - 1}) $ has been defined for some $ n \in \N $.  Note that $  \Sigma^n = \bigcup_{i = 1}^3 \{ i \, \omega \colon \omega \in \Sigma^{n - 1} \, \} $.  Let $ (iu, iv) \in U^n $ if $ (u, v) \in U^{n - 1} $ for $ u, v \in \Sigma^{n - 1} $.  For each distinct pair $ k, l \in \Sigma $, we add three further edges $ (lk^{n - 1}, kl^{n - 1}) \in U^n $.   We define $ \Gamma^n \coloneqq (\Sigma^n, U^n) $. This procedure is illustrated in Figure~\ref{fig:Graphs_Gamma_n}. If $ (u, v) \in U^n $, then we call the states $ u,v \in \Sigma^n $ \textsl{neighbours}, and write $ u \sim v $.
	
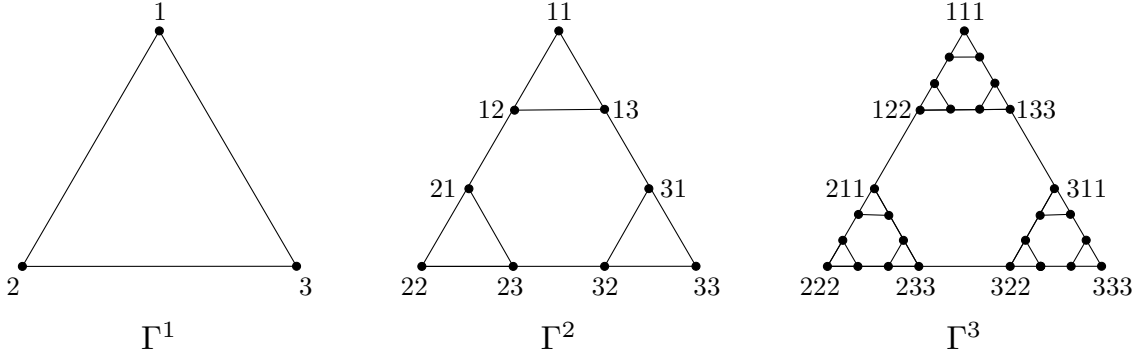
\begin{figure}
	\begin{minipage}{0.33\textwidth}
		\centering
		\begin{tikzpicture}[line cap=round,line join=round,>=triangle 45,x=1.0cm,y=1.0cm, scale=1, every node/.style={scale=1.}]
		\clip(2,0.2) rectangle (6.8,5.1);
		
		\draw [color=black] (2.54,1.42)-- (6.1455508529661005,1.421745472830631);
		\draw [color=black] (6.1455508529661005,1.421745472830631)-- (4.34126380267011,4.54337136972061);
		\draw [color=black] (4.34126380267011,4.54337136972061)-- (2.54,1.42);
		\draw [fill=black] (2.54,1.42) circle (1.5pt);
		\draw[color=black] (2.42,1.15) node {$ 2 $};
		
		\draw [fill=black] (6.1455508529661005,1.421745472830631) circle (1.5pt);
		\draw[color=black] (6.27,1.15) node {$ 3 $};
		
		\draw [fill=black] (4.34126380267011,4.54337136972061) circle (1.5pt);
		\draw[color=black] (4.35,4.8) node {$ 1 $};
		
		\draw[color=black] (4.34,0.5) node[scale=1.25] {$ \Gamma^1 $};
		\end{tikzpicture}
	\end{minipage}\hfill
	\begin{minipage}{0.33\textwidth}
		\centering

		\begin{tikzpicture}[line cap=round,line join=round,>=triangle 45,x=1.0cm,y=1.0cm, scale=1, every node/.style={scale=1.}]
		\clip(2,0.2) rectangle (6.8,5.1);
		
		\draw [color=black] (2.54,1.42)-- (6.1455508529661005,1.421745472830631);
		\draw [color=black] (6.1455508529661005,1.421745472830631)-- (4.34126380267011,4.54337136972061);
		\draw [color=black] (4.34126380267011,4.54337136972061)-- (2.54,1.42);
		\draw (2.54,1.42)-- (3.741850425154662,1.42);
		\draw (3.1591384971348377,2.450102890885205)-- (3.741850425154662,1.42);
		\draw (4.94370099114192,1.420581824413233)-- (5.525415285984578,2.4512484223207944);
		\draw (3.759755878631381,3.4915667401792634)-- (4.9431964826727395,3.5031206820547505);
		\draw [fill=black] (2.54,1.42) circle (1.5pt);
		\draw[color=black] (2.44,1.15) node {$ 22 $};
		
		\draw [fill=black] (6.1455508529661005,1.421745472830631) circle (1.5pt);
		\draw[color=black] (6.29,1.15) node {$ 33 $};
		
		\draw [fill=black] (4.34126380267011,4.54337136972061) circle (1.5pt);
		\draw[color=black] (4.35,4.8) node {$ 11 $};
		
		\draw [fill=black] (3.741850425154662,1.42) circle (1.5pt);
		\draw[color=black] (3.7,1.15) node {$ 23 $};
		
		\draw [fill=black] (3.1591384971348377,2.450102890885205) circle (1.5pt);
		\draw[color=black] (2.82,2.46) node {$ 21 $};
		
		\draw [fill=black] (5.525415285984578,2.4512484223207944) circle (1.5pt);
		\draw[color=black] (5.85,2.45) node {$ 31 $};
		
		\draw [fill=black] (4.94370099114192,1.420581824413233) circle (1.5pt);
		\draw[color=black] (4.95,1.15) node {$ 32 $};
		
		\draw [fill=black] (3.759755878631381,3.4915667401792634) circle (1.5pt);
		\draw[color=black] (3.45,3.5) node {$ 12 $};
		
		\draw [fill=black] (4.9431964826727395,3.5031206820547505) circle (1.5pt);
		\draw[color=black] (5.22,3.5) node {$ 13 $};
		
		\draw[color=black] (4.34,0.5) node[scale=1.25] {$ \Gamma^2 $};
		\end{tikzpicture}
	\end{minipage}\hfill
	\begin{minipage}{0.34\textwidth}
		\centering
		\begin{tikzpicture}[line cap=round,line join=round,>=triangle 45,x=1.0cm,y=1.0cm, scale=1, every node/.style={scale=1.}]
		\clip(2,0.2) rectangle (6.8,5.1);
		
		\draw [color=black] (2.54,1.42)-- (6.1455508529661005,1.421745472830631);
		\draw [color=black] (6.1455508529661005,1.421745472830631)-- (4.34126380267011,4.54337136972061);
		\draw [color=black] (4.34126380267011,4.54337136972061)-- (2.54,1.42);
		\draw (2.54,1.42)-- (3.741850425154662,1.42);
		\draw (3.1591384971348377,2.450102890885205)-- (3.741850425154662,1.42);
		\draw (4.94370099114192,1.420581824413233)-- (5.525415285984578,2.4512484223207944);
		\draw (3.759755878631381,3.4915667401792634)-- (4.9431964826727395,3.5031206820547505);
		
		\draw [fill=black] (2.54,1.42) circle (1.5pt);
		\draw[color=black] (2.44,1.15) node {$ 222 $};
		
		\draw [fill=black] (6.1455508529661005,1.421745472830631) circle (1.5pt);
		\draw[color=black] (6.29,1.15) node {$ 333 $};
		
		\draw [fill=black] (4.34126380267011,4.54337136972061) circle (1.5pt);
		\draw[color=black] (4.35,4.8) node {$ 111 $};
		
		\draw [fill=black] (3.741850425154662,1.42) circle (1.5pt);
		\draw[color=black] (3.7,1.15) node {$ 233 $};
		
		\draw [fill=black] (3.1591384971348377,2.450102890885205) circle (1.5pt);
		\draw[color=black] (2.78,2.45) node {$ 211 $};
		
		\draw [fill=black] (5.525415285984578,2.4512484223207944) circle (1.5pt);
		\draw[color=black] (5.95,2.45) node {$ 311 $};
		
		\draw [fill=black] (4.94370099114192,1.420581824413233) circle (1.5pt);
		\draw[color=black] (4.95,1.15) node {$ 322 $}; 
		
		\draw [fill=black] (3.759755878631381,3.4915667401792634) circle (1.5pt);
		\draw[color=black] (3.383,3.5) node {$ 122 $};
		
		\draw [fill=black] (4.9431964826727395,3.5031206820547505) circle (1.5pt);
		\draw[color=black] (5.283,3.5) node {$ 133 $}; 
		
		\draw [fill=black] (2.9406168083848874,1.42) circle (1.5pt);
		\draw [fill=black] (2.7420495149418573,1.7659332604307565) circle (1.5pt);
		\draw [fill=black] (2.948461587479838,2.109355085464961) circle (1.5pt);
		\draw [fill=black] (3.3471207044017284,2.0963282211990313) circle (1.5pt);
		\draw [fill=black] (3.3412336167697747,1.42) circle (1.5pt);
		\draw [fill=black] (3.539800910212805,1.7659332604307565) circle (1.5pt);
		\draw [fill=black] (5.9431664953000976,1.7674829439450774) circle (1.5pt);
		\draw [fill=black] (5.744934232358041,1.4213575900248316) circle (1.5pt);
		\draw [fill=black] (5.344317611749981,1.4209697072190324) circle (1.5pt);
		\draw [fill=black] (5.337775696419512,2.097291911449265) circle (1.5pt);
		\draw [fill=black] (5.344317611749981,1.4209697072190324) circle (1.5pt);
		\draw [fill=black] (5.344317611749981,1.4209697072190324) circle (1.5pt);
		\draw [fill=black] (5.736422013689717,2.1107047569507538) circle (1.5pt);
		\draw [fill=black] (5.14541547395062,1.7667105498691136) circle (1.5pt);
		\draw [fill=black] (3.949171980087133,3.844575759426359) circle (1.5pt);
		\draw [fill=black] (4.160111003663561,3.5060443442959937) circle (1.5pt);
		\draw [fill=black] (4.143038548477811,4.195232639096031) circle (1.5pt);
		\draw [fill=black] (4.541908029337653,4.196621140498657) circle (1.5pt);
		\draw [fill=black] (4.742552256005196,3.8498709112767027) circle (1.5pt);
		\draw [fill=black] (4.542582972044068,3.5047461857503825) circle (1.5pt);
		\draw (2.54,1.42)-- (3.741850425154662,1.42);
		\draw (3.1591384971348377,2.450102890885205)-- (3.741850425154662,1.42);
		\draw (4.94370099114192,1.420581824413233)-- (5.525415285984578,2.4512484223207944);
		\draw (3.759755878631381,3.4915667401792634)-- (4.9431964826727395,3.5031206820547505);
		\draw (2.54,1.42)-- (2.9406168083848874,1.42);
		\draw (2.54,1.42)-- (2.7420495149418573,1.7659332604307565);
		\draw (2.7420495149418573,1.7659332604307565)-- (2.9406168083848874,1.42);
		\draw (3.949171980087133,3.844575759426359)-- (4.160111003663561,3.5060443442959937);
		\draw (4.542582972044068,3.5047461857503825)-- (4.742552256005196,3.8498709112767027);
		\draw (4.541908029337653,4.196621140498657)-- (4.143038548477811,4.195232639096031);
		\draw (5.337775696419512,2.097291911449265)-- (5.736422013689717,2.1107047569507538);
		\draw (5.9431664953000976,1.7674829439450774)-- (5.744934232358041,1.4213575900248316);
		\draw (5.344317611749981,1.4209697072190324)-- (5.14541547395062,1.7667105498691136);
		\draw (3.539800910212805,1.7659332604307565)-- (3.3412336167697747,1.42);
		\draw (3.3471207044017284,2.0963282211990313)-- (2.948461587479838,2.109355085464961);
		
		\draw[color=black] (4.34,0.5) node[scale=1.25] {$ \Gamma^3 $};
		\end{tikzpicture} 
		
	\end{minipage}
	\caption{The Graphs $ \Gamma^1 $, $ \Gamma^2 $ and $ \Gamma^3 $.} \label{fig:Graphs_Gamma_n}
\end{figure}
	
	\subsection{The standard projection} \label{Standard_projection}
	
	Let $ q_1 = (1/2, \sqrt{3}/2) $,  $ q_2 = (0,0) $ and $ q_3 = (1,0) $. For $i \in \Sigma$, define $ S_i \colon \R^2 \to \R^2 $ by $ S_i(x) \coloneqq \frac{1}{2} \, (x + q_i)$.  The Sierpi\'nski gasket is the attractor of the iterated function system $ (\R^2; S_1, S_2, S_3) $, that is the unique non-empty compact set $ \mathcal{K} $ satisfying $ \mathcal{K} = S_1(\mathcal{K}) \cup S_2(\mathcal{K}) \cup S_3(\mathcal{K})$;  see \cite{KF:1990} for further details.  For $ m \in \N $ and $ \omega = \omega_1 \ldots \omega_m \in \Sigma^m $ define $ S_{\omega} \coloneqq S_{\omega_1} \circ \cdots \circ S_{\omega_m} $ and for $ \omega = \vartheta $ set $ S_{\omega} \coloneqq \text{id} $.  Notice, $ u \sim v $ is equivalent to $ S_u(\mathcal{K}) \cap S_v(\mathcal{K}) \neq \emptyset$.
	
	To prove that the Martin boundary is homeomorphic to the Sierpi\'nski gasket, we use the \textsl{standard projection} $ \pi \colon \Sinf \to \mathcal{K} $ defined by $ \pi(\mbf x) = \lim_{n \to \infty} S_{i_1} \circ \cdots \circ S_{i_n}(x_0) $ for $ \mbf x = i_1 i_2 \ldots  \in \Sinf $. Here, $ x_0 \in \R^2 $ is arbitrary and the definition of $ \pi $ is independent of the choice of $ x_0 $. Two states $ \mbf{x,y} \in \Sinf $ are called \textsl{$ \pi $-equivalent}, denoted by $ \mbf x \aq{\pi} \mbf y $, if $ \pi(\mbf x ) = \pi(\mbf y) $.  Two distinct states $ \mbf x = i_1 i_2 \ldots \in \Sinf $ and $ \mbf y = j_1 j_2 \ldots \in \Sinf $ are $ \pi $-equivalent if and only if there exist a $ m \in \N_0 $ such that $ i_p = j_p $ for all $ p \in \{1, \ldots, m \} $ and $ i_{m + 1} i_{m + 2} \ldots = l k^{\infty} $ and $ j_{m + 1} j_{m + 2} \ldots = k l^{\infty} $ for some $ k,l \in \Sigma $ distinct. 
	
	\subsection{The Markov chain} \label{The_Markov_chain}
	
	Throughout this section let $n, k \in \N$ with $k \leq n$ and $ \omega \in \Sigma^{k-1} $ be fixed.  Each vertex $ u = \omega ij^{n - k} \in \tilde{\Sigma}^n $  has three neighbouring vertices in $ \Sigma^n $ and is a junction point of three edges in $ \Gamma^n $, of which two are lying `on a line', see Figure~\ref{fig:bild-mitte}. This can be considered as $ u $ connecting a short and a long line in the graph $ \Gamma^n $.  Formally, the neighbours of $ u $ are the other two vertices in the $ n $-cell containing $ u $, that is $ z = \omega i j^{n - k - 1} l $ and $ v = \omega i j^{n - k - 1} i $ with pairwise distinct $ i, j, l \in \Sigma $, and the word $ \pi $-equivalent to $ u $, that is $ w = \omega j i^{n-k} $.  The two neighbouring vertices of $ u $ `on a line' are $ v $ and the $ \pi $-equivalent word $ w $, see Figure~\ref{fig:bild-rechts}. 
	
	We assign one probability to stay on the outer part of a cell in $ \Gamma^n $ and another to walk into a deeper cell of $ \Gamma^n $. Namely, for $ p \in (0,1/2) $ and $ q \coloneqq 1 - 2 \, p $, we let $ (X_n)_{n \in \N_0} $ denote the Markov chain with origin $ \vartheta $, state space $ \Sigma^* $ and transition probability matrix $P$ given by, $ P(\vartheta, i) \coloneqq 1/3 $ for each $ i \in \Sigma $ and
	\begin{equation*}
	P(u,v) \coloneqq 
	\begin{cases}
	p & \text{ if } u = \omega ij^{n-k} \in \tilde{\Sigma}^n, v \in \Sigma^n \text{ and } u \aq{\pi} v \text{ or } v = \omega i j^{n-k-1} i \text{ for distinct } i, j \in \Sigma, \\
	q & \text{ if } u = \omega ij^{n-k} \in \tilde{\Sigma}^n, v \in \Sigma^n \text{ and } v = \omega i j^{n-k-1} l \text{ for pairwise distinct } i, j, l \in \Sigma, \\
	1/3 & \text{ if } u \in V^n \text{ and } v = ui \text{ for } i \in \Sigma, \\
	0 & \text{ otherwise}.
	\end{cases}
	\end{equation*}
	For $ p = 1/3 $ the above Markov chain coincides with the one in \cite{Lau}. If the chain starts at a word in $ \tilde{\Sigma}^n $, then it walks to one of its three neighbours, see Figures~\ref{fig:bild-mitte} and~\ref{fig:bild-rechts}.  If the chain hits an element $ u \in V^n $, then it moves to $ ui \in V^{n + 1} $, $ i \in \Sigma $, on the next level, see Figure~\ref{fig:bild-links}. Thus, the three boundary vertices $ V^n $ of $ \Gamma^n $ play an important role in the definition of our Markov chain as they act similar to an absorbing state, meaning once the chain hits one of these vertices, it must move to the next level and cannot return to prior states.
	
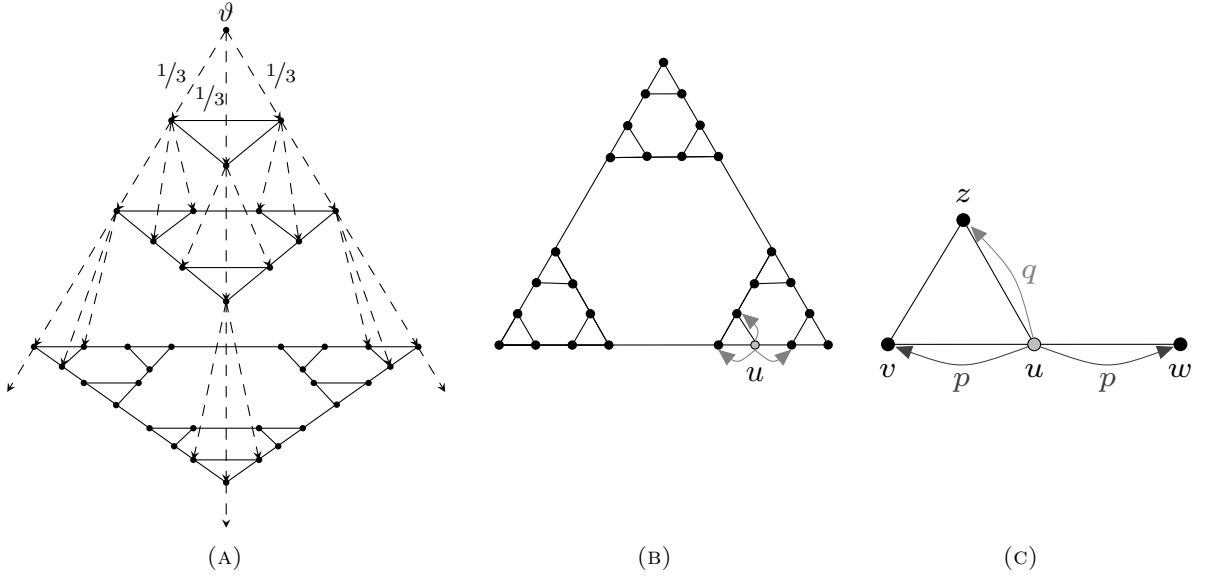
\begin{figure}
	\centering %
	\begin{subfigure}[b]{0.415\textwidth}
		\centering
		\begin{tikzpicture}[line cap=round,line join=round,>=triangle 45,x=.9cm,y=.75cm, scale=0.8, every node/.style={scale=1.}]
		\clip(-4.5,-1) rectangle (4.5,11);
		\node (0) at (0,10) {};
		\fill [color=black] (0,10) circle (1.5pt);
		\node (1) at (0,7) {};
		\fill [color=black] (0,7) circle (1.5pt);
		\node (2) at (1,8) {};
		\fill [color=black] (1,8) circle (1.5pt);
		\node (3) at (-1,8) {};
		\fill [color=black] (-1,8) circle (1.5pt);
		\node (11) at (0,4) {};
		\fill [color=black] (0,4) circle (1.5pt);
		\node (22) at (2,6) {};
		\fill [color=black] (2,6) circle (1.5pt);
		\node (33) at (-2,6) {};
		\fill [color=black] (-2,6) circle (1.5pt);
		\node (12) at (.8,4.75) {};
		\fill [color=black] (0.8,4.75) circle (1.5pt);
		\node (13) at (-.8,4.75) {};
		\fill [color=black] (-0.8,4.75) circle (1.5pt);
		\node (21) at (1.33,5.33) {};
		\fill [color=black] (1.33,5.33) circle (1.5pt);
		\node (23) at (.6,6) {};
		\fill [color=black] (.6,6) circle (1.5pt);
		\node (32) at (-.6,6) {};
		\fill [color=black] (-.6,6) circle (1.5pt);
		\node (31) at (-1.33,5.33) {};
		\fill [color=black] (-1.33,5.33) circle (1.5pt);
		\node (111) at (0,0) {};
		\fill [color=black] (0,0) circle (1.5pt);
		\node (112) at (0.6,0.5) {};
		\fill [color=black] (0.6,0.5) circle (1.5pt);
		\node (113) at (-0.6,.5) {};
		\fill [color=black] (-0.6,0.5) circle (1.5pt);
		\node (121) at (.95,.8) {};
		\fill [color=black] (0.95,0.8) circle (1.5pt);
		\node (122) at (1.4,1.2) {};
		\fill [color=black] (1.4,1.2) circle (1.5pt);
		\node (123) at (.6,1.2) {};
		\fill [color=black] (0.6,1.2) circle (1.5pt);
		\node (131) at (-.95,.8) {};
		\fill [color=black] (-0.95,0.8) circle (1.5pt);
		\node (132) at (-0.6,1.2) {};
		\fill [color=black] (-0.6,1.2) circle (1.5pt);
		\node (133) at (-1.4,1.2) {};
		\fill [color=black] (-1.4,1.2) circle (1.5pt);
		\node (211) at (2,1.7) {};
		\fill [color=black] (2.02,1.72) circle (1.5pt);
		\node (212) at (2.6,2.2) {};
		\fill [color=black] (2.6,2.2) circle (1.5pt);
		\node (213) at (1.6,2.2) {};
		\fill [color=black] (1.6,2.2) circle (1.5pt);
		\node (221) at (3,2.56) {};
		\fill [color=black] (3,2.56) circle (1.5pt);
		\node (222) at (3.5,3) {};
		\fill [color=black] (3.51,3) circle (1.5pt);
		\node (223) at (2.6,3) {};
		\fill [color=black] (2.6,3) circle (1.5pt);
		\node (231) at (1.4,2.5) {};
		\fill [color=black] (1.4,2.5) circle (1.5pt);
		\node (232) at (1.8,3) {};
		\fill [color=black] (1.8,3) circle (1.5pt);
		\node (233) at (1,3) {};
		\fill [color=black] (1,3) circle (1.5pt);
		\node (311) at (-2,1.7) {};
		\fill [color=black] (-2.01,1.71) circle (1.5pt);
		\node (312) at (-1.6,2.2) {};
		\fill [color=black] (-1.6,2.2) circle (1.5pt);
		\node (313) at (-2.6,2.2) {};
		\fill [color=black] (-2.6,2.2) circle (1.5pt);
		\node (321) at (-1.4,2.5) {};
		\fill [color=black] (-1.4,2.5) circle (1.5pt);
		\node (322) at (-1,3) {};
		\fill [color=black] (-1,3) circle (1.5pt);
		\node (323) at (-1.8,3) {};
		\fill [color=black] (-1.8,3) circle (1.5pt);
		\node (331) at (-3,2.56) {};
		\fill [color=black] (-3,2.56) circle (1.5pt);
		\node (332) at (-2.6,3) {};
		\fill [color=black] (-2.6,3) circle (1.5pt);
		\node (333) at (-3.5,3) {};
		\fill [color=black] (-3.51,3) circle (1.5pt);

		\draw (0,7) -- (1,8) -- (-1,8) -- (0,7);
		\draw (0,4) -- (2,6) -- (-2,6) -- (0,4);
		\draw (0,0) -- (3.5,3) -- (-3.5,3) -- (0,0);
		
		\draw (-1.33,5.33) -- (-.6,6);
		\draw (1.33,5.33) -- (.6,6);
		\draw (-.8,4.75) -- (.8,4.75);
		
		\draw (-0.6,.5) -- (0.6,.5);
		\draw (.6,1.2) -- (.95,.8);
		\draw (-.6,1.2) -- (-.95,.8);
		
		\draw (2.6,2.2) -- (1.6,2.2);
		\draw (2.6,3) -- (3,2.56);
		\draw (1.4,2.5) -- (1.8,3);
		
		\draw (-2.6,2.2) -- (-1.6,2.2);
		\draw (-2.6,3) -- (-3,2.56);
		\draw (-1.4,2.5) -- (-1.8,3);
		
		\draw (-2,1.7)-- (-1,3);
		\draw (1,3)-- (2.,1.7);
		\draw (-1.4,1.2)-- (1.4,1.2);
		
		\draw [>=stealth, ->, dash pattern=on 4pt off 4pt] (0,10)-- (0,7);
		\draw [>=stealth, ->, dash pattern=on 4pt off 4pt] (0,10)-- (1,8);
		\draw [>=stealth, ->, dash pattern=on 4pt off 4pt] (0,10)-- (-1,8);
		
		\draw [>=stealth, ->, dash pattern=on 4pt off 4pt] (0,7)-- (0,4);
		\draw [>=stealth, ->, dash pattern=on 4pt off 4pt] (1,8)-- (2,6);
		\draw [>=stealth, ->, dash pattern=on 4pt off 4pt] (-1,8)-- (-2,6);
		
		\draw [>=stealth, ->, dash pattern=on 4pt off 4pt] (0,4)-- (0,0);
		\draw [>=stealth, ->, dash pattern=on 4pt off 4pt] (2,6)-- (3.5,3);
		\draw [>=stealth, ->, dash pattern=on 4pt off 4pt] (-2,6)-- (-3.5,3);
		
		\draw [>=stealth, ->, dash pattern=on 4pt off 4pt] (0,0)-- (0,-1);
		\draw [>=stealth, ->, dash pattern=on 4pt off 4pt] (3.5,3)-- (4,2);
		\draw [>=stealth, ->, dash pattern=on 4pt off 4pt] (-3.5,3)-- (-4,2);
		
		\draw [>=stealth, ->, dash pattern=on 4pt off 4pt] (0,7)-- (.8,4.75);
		\draw [>=stealth, ->, dash pattern=on 4pt off 4pt] (0,7)-- (-.8,4.75);
		
		\draw [>=stealth, ->, dash pattern=on 4pt off 4pt] (-1,8)-- (-1.33,5.33);
		\draw [>=stealth, ->, dash pattern=on 4pt off 4pt] (-1,8)-- (-.6,6);
		
		\draw [>=stealth, ->, dash pattern=on 4pt off 4pt] (1,8)-- (1.33,5.33);
		\draw [>=stealth, ->, dash pattern=on 4pt off 4pt] (1,8)-- (.6,6);
		
		\draw [>=stealth, ->, dash pattern=on 4pt off 4pt] (-2,6)-- (-3,2.56);
		\draw [>=stealth, ->, dash pattern=on 4pt off 4pt] (-2,6)-- (-2.6,3);
		
		\draw [>=stealth, ->, dash pattern=on 4pt off 4pt] (2,6)-- (3,2.56);
		\draw [>=stealth, ->, dash pattern=on 4pt off 4pt] (2,6)-- (2.6,3);
		
		\draw [>=stealth, ->, dash pattern=on 4pt off 4pt] (0,4)-- (0.6,0.5);
		\draw [>=stealth, ->, dash pattern=on 4pt off 4pt] (0,4)-- (-0.6,0.5);

		\node (theta) at (0,10.4) {$\vartheta$};
		\node (p1) at (-1,9) {$\nicefrac{1}{3}$};
		\node (p2) at (1,9) {$\nicefrac{1}{3}$};
		\node (p1) at (-.3,8.5) {$\nicefrac{1}{3}$};
		\end{tikzpicture}
		\caption{}\label{fig:bild-links}
	\end{subfigure} %
	\begin{subfigure}[b]{0.28\textwidth}
		\centering
		\begin{tikzpicture}[line cap=round,line join=round,>=triangle 45,x=1.0cm,y=1.0cm, scale=1.2, every node/.style={scale=1.2}]
		\clip(2.4,-0.6) rectangle (6.4,5);
		
		\draw [color=black] (2.54,1.42)-- (6.1455508529661005,1.421745472830631);
		\draw [color=black] (6.1455508529661005,1.421745472830631)-- (4.34126380267011,4.54337136972061);
		\draw [color=black] (4.34126380267011,4.54337136972061)-- (2.54,1.42);
		\draw (2.54,1.42)-- (3.741850425154662,1.42);
		\draw (3.1591384971348377,2.450102890885205)-- (3.741850425154662,1.42);
		\draw (4.94370099114192,1.420581824413233)-- (5.525415285984578,2.4512484223207944);
		\draw (3.759755878631381,3.4915667401792634)-- (4.9431964826727395,3.5031206820547505);
		
		\draw [fill=black] (2.54,1.42) circle (1.3pt);
		
		\draw [fill=black] (6.1455508529661005,1.421745472830631) circle (1.3pt);
		
		\draw [fill=black] (4.34126380267011,4.54337136972061) circle (1.3pt);
		
		\draw [fill=black] (3.741850425154662,1.42) circle (1.3pt);
		
		\draw [fill=black] (3.1591384971348377,2.450102890885205) circle (1.3pt);
		
		\draw [fill=black] (5.525415285984578,2.4512484223207944) circle (1.3pt);
		
		\draw [fill=black] (4.94370099114192,1.420581824413233) circle (1.3pt);
		
		\draw [fill=black] (3.759755878631381,3.4915667401792634) circle (1.3pt);
		
		\draw [fill=black] (4.9431964826727395,3.5031206820547505) circle (1.3pt);
		
		\draw [fill=black] (2.9406168083848874,1.42) circle (1.3pt);
		\draw [fill=black] (2.7420495149418573,1.7659332604307565) circle (1.3pt);
		\draw [fill=black] (2.948461587479838,2.109355085464961) circle (1.3pt);
		\draw [fill=black] (3.3471207044017284,2.0963282211990313) circle (1.3pt);
		\draw [fill=black] (3.3412336167697747,1.42) circle (1.3pt);
		\draw [fill=black] (3.539800910212805,1.7659332604307565) circle (1.3pt);
		\draw [fill=black] (5.9431664953000976,1.7674829439450774) circle (1.3pt);
		\draw [fill=black] (5.744934232358041,1.4213575900248316) circle (1.3pt);
		
		\draw [fill=lightgray] (5.344317611749981,1.4209697072190324) circle (1.3pt);
		\node (1) at (5.344317611749981,1.1) {$ u $};
		
		\draw [fill=black] (5.337775696419512,2.097291911449265) circle (1.3pt);
		\draw [fill=black] (5.736422013689717,2.1107047569507538) circle (1.3pt);
		\draw [fill=black] (5.14541547395062,1.7667105498691136) circle (1.3pt);
		\draw [fill=black] (3.949171980087133,3.844575759426359) circle (1.3pt);
		\draw [fill=black] (4.160111003663561,3.5060443442959937) circle (1.3pt);
		\draw [fill=black] (4.143038548477811,4.195232639096031) circle (1.3pt);
		\draw [fill=black] (4.541908029337653,4.196621140498657) circle (1.3pt);
		\draw [fill=black] (4.742552256005196,3.8498709112767027) circle (1.3pt);
		\draw [fill=black] (4.542582972044068,3.5047461857503825) circle (1.3pt);
		\draw (2.54,1.42)-- (3.741850425154662,1.42);
		\draw (3.1591384971348377,2.450102890885205)-- (3.741850425154662,1.42);
		\draw (4.94370099114192,1.420581824413233)-- (5.525415285984578,2.4512484223207944);
		\draw (3.759755878631381,3.4915667401792634)-- (4.9431964826727395,3.5031206820547505);
		\draw (2.54,1.42)-- (2.9406168083848874,1.42);
		\draw (2.54,1.42)-- (2.7420495149418573,1.7659332604307565);
		\draw (2.7420495149418573,1.7659332604307565)-- (2.9406168083848874,1.42);
		\draw (3.949171980087133,3.844575759426359)-- (4.160111003663561,3.5060443442959937);
		\draw (4.542582972044068,3.5047461857503825)-- (4.742552256005196,3.8498709112767027);
		\draw (4.541908029337653,4.196621140498657)-- (4.143038548477811,4.195232639096031);
		\draw (5.337775696419512,2.097291911449265)-- (5.736422013689717,2.1107047569507538);
		\draw (5.9431664953000976,1.7674829439450774)-- (5.744934232358041,1.4213575900248316);
		\draw (5.344317611749981,1.46)-- (5.14541547395062,1.7667105498691136);
		\draw (3.539800910212805,1.7659332604307565)-- (3.3412336167697747,1.42);
		\draw (3.3471207044017284,2.0963282211990313)-- (2.948461587479838,2.109355085464961);

		\draw[gray, ->] (5.345,1.365) .. controls (5.544,1.2) .. (5.744,1.37);
		\draw[gray, ->] (5.345,1.365) .. controls (5.15,1.2) .. (4.94,1.37);
		\draw[gray, ->] (5.345,1.46) .. controls (5.4,1.6) .. (5.2,1.77);

		\end{tikzpicture} 
		\caption{} \label{fig:bild-mitte}
	\end{subfigure}
	\hfill
	\begin{subfigure}[b]{0.27\textwidth}
		\centering
		\begin{tikzpicture}[line cap=round,line join=round,>=triangle 45,x=1.0cm,y=1.0cm, scale=1.6, every node/.style={scale=1.6}]
		\clip(2.3,-.1) rectangle (5.2,3);
		\draw [color=black] (3.1591384971348377,2.450102890885205)-- (2.54,1.42);
		
		\draw (2.54,1.42)-- (4.94370099114192,1.420581824413233);
		
		\draw (3.1591384971348377,2.450102890885205)-- (3.741850425154662,1.42);

		\begin{scriptsize}
		\draw [fill=black] (2.54,1.42) circle (1.5pt);
		\node (1) at (2.54,1.2) {$ v $};
		\draw [fill=lightgray] (3.741850425154662,1.42) circle (1.5pt);
		\node (2) at (3.741850425154662,1.2) {$ u $};
		\draw [fill=black] (4.94370099114192,1.420581824413233) circle (1.5pt);
		\node (3) at (4.94370099114192,1.2) {$ w $};
		\draw [fill=black] (3.1591384971348377,2.450102890885205) circle (1.5pt);
		\node (4) at (3.1591384971348377,2.65) {$ z $};
		\draw[darkgray, <-] (2.6,1.4) .. controls (3.145,1.2) .. (3.69,1.4) node at (3.145,1.1) {$ p $};
		\draw[darkgray, ->] (3.79,1.4) .. controls (4.345,1.2) .. (4.89,1.4) node at (4.345,1.1) {$ p $};
		\draw[gray, ->] (3.73,1.485) .. controls (3.6,2.) .. (3.22,2.4) node at (3.7,2.) {$ q $};
		\end{scriptsize}
		\end{tikzpicture}
		\caption{}\label{fig:bild-rechts}
	\end{subfigure} 
	\caption{Transition probabilities of the Markov chain.} \label{fig:P}
\end{figure}
	
\section{General framework} \label{sec_frame}
	
This section is an overview of the framework given in \cite{Lau}, highlighting relevant changes in definitions, statement of results and proofs.
	
\subsection{Hitting probabilities and random matrix product} \label{Random_matrix_product}
	
We denote the probability, conditioned on starting at a state $ x \in \Sigma^* $, to eventually arrive at a state $y \in \Sigma^*$ by $\rho_{x, y} \coloneqq \P(\, \exists \, k \in \N_0 : X_k = y \, | \, X_0 = x) $.  In this section, for a given $n \in \N$ and $ x \in \Sigma^n $, we are concerned with computing $ \rho_i(x) \coloneqq \rho_{x, i^n} $, namely the probability to be absorbed by $ i^n $ when starting at $ x $.  To this end we define $\boldsymbol{\rho} \colon \Sigma^* \to [0, 1]^{3}$ by $ \boldsymbol{\rho}(x) \coloneqq [\rho_1(x), \rho_2(x), \rho_3(x)] $; that is the vector with the probabilities to be absorbed by one of the three vertices of $ V^n $ when starting at $ x \in \Sigma^n $. Note that $ \rho_i(j^n) = \delta_{ji} $ for $ i, j \in \Sigma $ and $n \in \N$.  Here, $ \delta_{ij} $ denotes the Kronecker delta symbol, namely $ \delta_{ij} = 1 $ if $ i = j $ and $ \delta_{ij} = 0 $ otherwise. 
	
On each level $ n \in \N $, the chain only moves to the next level if it reaches an element of $ V^{n} $. If the chain then lands in the interior of $ \Sigma^{n + 1} $, the  probability to reach a vertex in $ V^{n + 1} $ is, by symmetry, given by $ a_{n+1} \coloneqq \rho_1(1^{n}2) $, $ b_{n+1} \coloneqq \rho_2(1^{n}2) $ and $ c_{n+1} \coloneqq \rho_3(1^{n}2) $.  In the sequel, let $ x = i_1 \ldots i_n \in \tilde{\Sigma}^n $ be fixed, where $n \in \N$ with $n \geq 2$.  We have that
	\begin{align} \label{eq_cells_containing_x}
	x \in \Delta_{i_1 \ldots i_{n - 2} i_{n - 1}}^n \subset \Delta_{i_1 \ldots i_{n - 2}}^n \subset \ldots \subset \Delta_{i_1}^n.
	\end{align}
	This means that, starting at $ x $ and reaching one of the vertices in $ V^n \subset \partial \Delta_{i_1}^n $, one first needs to pass one of the three outer vertices of $ \Delta_{i_1 \ldots i_{n - 2}}^n $, then one of the outer vertices of $ \Delta_{i_1 \ldots i_{n - 3}}^n $ and so forth going through one of the outer vertices of each of the cells in \eqref{eq_cells_containing_x}.  To calculate $ \boldsymbol{\rho}(x) $, we look at the probabilities to hit the outer vertices of $ \Delta_{i_1 \ldots i_{k}}^n \supset \Delta_{i_1 \ldots i_{k + 1}}^n $ when starting at one of the outer vertices of $ \Delta_{i_1 \ldots i_{k + 1}}^n $ for each $ k \leq n - 2 $.  These are, by symmetry, the probabilities for the chain starting at $ 12^{k - 1} $, to reach the vertices  $ 1^k, 2^k, 3^k $. For ease of notation, set $ \alpha_n \coloneqq \rho_1(12^{n - 1}) $, $ \beta_n \coloneqq \rho_2(12^{n - 1}) $ and $ \gamma_n \coloneqq \rho_3(12^{n - 1}) $.  Figure \ref{fig:a_n} shows $ a_n, b_n, c_n $ and $ \alpha_n, \beta_n, \gamma_n $ for $ n = 2 $ and $ n = 3 $.
	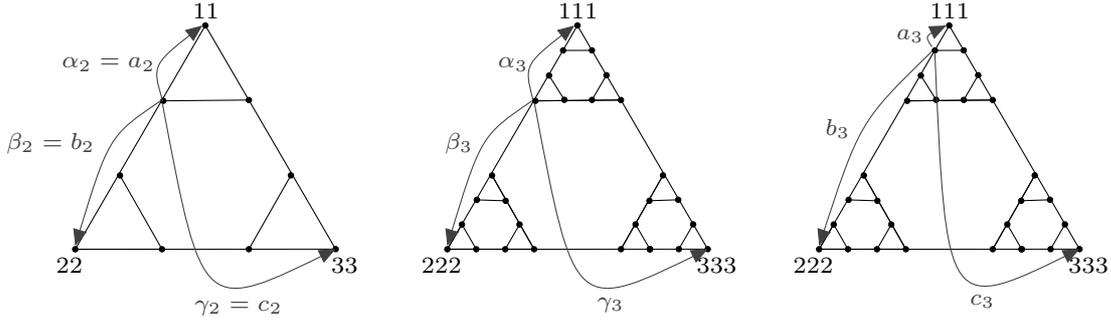
\begin{figure}
		\hspace{-2em}\begin{minipage}{0.3\textwidth}
			\centering
			\begin{tikzpicture}[line cap=round,line join=round,>=triangle 45,x=0.95cm,y=0.95cm, scale=1, every node/.style={scale=1.25}]
			\clip(1.5,0.4) rectangle (6.8,5);
			
			\draw [color=black] (2.54,1.42)-- (6.1455508529661005,1.421745472830631);
			\draw [color=black] (6.1455508529661005,1.421745472830631)-- (4.34126380267011,4.54337136972061);
			\draw [color=black] (4.34126380267011,4.54337136972061)-- (2.54,1.42);
			\draw (2.54,1.42)-- (3.741850425154662,1.42);
			\draw (3.1591384971348377,2.450102890885205)-- (3.741850425154662,1.42);
			\draw (4.94370099114192,1.420581824413233)-- (5.525415285984578,2.4512484223207944);
			\draw (3.759755878631381,3.4915667401792634)-- (4.9431964826727395,3.5031206820547505);
			\begin{scriptsize}
			\draw [fill=black] (2.54,1.42) circle (1pt);
			\draw[color=black] (2.46,1.2) node {$ 22 $};
			
			\draw [fill=black] (6.1455508529661005,1.421745472830631) circle (1pt);
			\draw[color=black] (6.27,1.2) node {$ 33 $};
			
			\draw [fill=black] (4.34126380267011,4.54337136972061) circle (1pt);
			\draw[color=black] (4.35,4.75) node {$ 11 $};
			
			\draw [fill=black] (3.741850425154662,1.42) circle (1pt);
			
			\draw [fill=black] (3.1591384971348377,2.450102890885205) circle (1pt);
			
			\draw [fill=black] (5.525415285984578,2.4512484223207944) circle (1pt);
			
			\draw [fill=black] (4.94370099114192,1.420581824413233) circle (1pt);
			
			\draw [fill=black] (3.759755878631381,3.4915667401792634) circle (1pt);
			
			\draw [fill=black] (4.9431964826727395,3.5031206820547505) circle (1pt);

			\draw[darkgray, ->] (3.75,3.495) .. controls (3.6,4) .. (4.3,4.54337136972061) node at (3,4) {$ \alpha_2 = a_2 $};
			
			\draw[darkgray, ->] (3.73,3.497) .. controls (3,3) .. (2.54,1.467) node at (2.2,2.9) {$ \beta_2 = b_2 $};
			
			\draw[darkgray, ->] (3.75,3.44) .. controls (4.3,0.5) .. (6.1455508529661005,1.421745472830631) node at (4.8,0.6) {$ \gamma_2 = c_2 $};
			
			\end{scriptsize}
			\end{tikzpicture}
		\end{minipage} 
		\begin{minipage}{0.3\textwidth}
			\centering
			\begin{tikzpicture}[line cap=round,line join=round,>=triangle 45,x=0.95cm,y=0.95cm, scale=1, every node/.style={scale=1.25}]
			\clip(1.5,0.4) rectangle (6.8,5);
						
			\draw [color=black] (2.54,1.42)-- (6.1455508529661005,1.421745472830631);
			\draw [color=black] (6.1455508529661005,1.421745472830631)-- (4.34126380267011,4.54337136972061);
			\draw [color=black] (4.34126380267011,4.54337136972061)-- (2.54,1.42);
			\draw (2.54,1.42)-- (3.741850425154662,1.42);
			\draw (3.1591384971348377,2.450102890885205)-- (3.741850425154662,1.42);
			\draw (4.94370099114192,1.420581824413233)-- (5.525415285984578,2.4512484223207944);
			\draw (3.759755878631381,3.4915667401792634)-- (4.9431964826727395,3.5031206820547505);
			\begin{scriptsize}
			\draw [fill=black] (2.54,1.42) circle (1pt);
			\draw[color=black] (2.46,1.2) node {$ 222 $};
			
			\draw [fill=black] (6.1455508529661005,1.421745472830631) circle (1pt);
			\draw[color=black] (6.27,1.2) node {$ 333 $};
			
			\draw [fill=black] (4.34126380267011,4.54337136972061) circle (1pt);
			\draw[color=black] (4.35,4.75) node {$ 111 $};
			
			\draw [fill=black] (3.741850425154662,1.42) circle (1pt);
			
			\draw [fill=black] (3.1591384971348377,2.450102890885205) circle (1pt);
			
			\draw [fill=black] (5.525415285984578,2.4512484223207944) circle (1pt);
			
			\draw [fill=black] (4.94370099114192,1.420581824413233) circle (1pt);
			
			\draw [fill=black] (3.759755878631381,3.4915667401792634) circle (1pt);
			
			\draw [fill=black] (4.9431964826727395,3.5031206820547505) circle (1pt);
			
			\draw [fill=black] (2.9406168083848874,1.42) circle (1pt);
			\draw [fill=black] (2.7420495149418573,1.7659332604307565) circle (1pt);
			\draw [fill=black] (2.948461587479838,2.109355085464961) circle (1pt);
			\draw [fill=black] (3.3471207044017284,2.0963282211990313) circle (1pt);
			\draw [fill=black] (3.3412336167697747,1.42) circle (1pt);
			\draw [fill=black] (3.539800910212805,1.7659332604307565) circle (1pt);
			\draw [fill=black] (5.9431664953000976,1.7674829439450774) circle (1pt);
			\draw [fill=black] (5.744934232358041,1.4213575900248316) circle (1pt);
			\draw [fill=black] (5.344317611749981,1.4209697072190324) circle (1pt);
			\draw [fill=black] (5.337775696419512,2.097291911449265) circle (1pt);
			\draw [fill=black] (5.344317611749981,1.4209697072190324) circle (1pt);
			\draw [fill=black] (5.344317611749981,1.4209697072190324) circle (1pt);
			\draw [fill=black] (5.736422013689717,2.1107047569507538) circle (1pt);
			\draw [fill=black] (5.14541547395062,1.7667105498691136) circle (1pt);
			\draw [fill=black] (3.949171980087133,3.844575759426359) circle (1pt);
			\draw [fill=black] (4.160111003663561,3.5060443442959937) circle (1pt);
			\draw [fill=black] (4.143038548477811,4.195232639096031) circle (1pt);
			\draw [fill=black] (4.541908029337653,4.196621140498657) circle (1pt);
			\draw [fill=black] (4.742552256005196,3.8498709112767027) circle (1pt);
			\draw [fill=black] (4.542582972044068,3.5047461857503825) circle (1pt);
			\draw (2.54,1.42)-- (3.741850425154662,1.42);
			\draw (3.1591384971348377,2.450102890885205)-- (3.741850425154662,1.42);
			\draw (4.94370099114192,1.420581824413233)-- (5.525415285984578,2.4512484223207944);
			\draw (3.759755878631381,3.4915667401792634)-- (4.9431964826727395,3.5031206820547505);
			\draw (2.54,1.42)-- (2.9406168083848874,1.42);
			\draw (2.54,1.42)-- (2.7420495149418573,1.7659332604307565);
			\draw (2.7420495149418573,1.7659332604307565)-- (2.9406168083848874,1.42);
			\draw (3.949171980087133,3.844575759426359)-- (4.160111003663561,3.5060443442959937);
			\draw (4.542582972044068,3.5047461857503825)-- (4.742552256005196,3.8498709112767027);
			\draw (4.541908029337653,4.196621140498657)-- (4.143038548477811,4.195232639096031);
			\draw (5.337775696419512,2.097291911449265)-- (5.736422013689717,2.1107047569507538);
			\draw (5.9431664953000976,1.7674829439450774)-- (5.744934232358041,1.4213575900248316);
			\draw (5.344317611749981,1.4209697072190324)-- (5.14541547395062,1.7667105498691136);
			\draw (3.539800910212805,1.7659332604307565)-- (3.3412336167697747,1.42);
			\draw (3.3471207044017284,2.0963282211990313)-- (2.948461587479838,2.109355085464961);

			\draw[darkgray, ->] (3.725,3.51) .. controls (3.6,4) .. (4.3,4.54337136972061) node at (3.45,4) {$ \alpha_3 $};
			
			\draw[darkgray, ->] (3.71,3.49) .. controls (3,3) .. (2.54,1.467) node at (2.7,2.9) {$ \beta_3 $};
			
			\draw[darkgray, ->] (3.75,3.44) .. controls (4.3,0.5) .. (6.1455508529661005,1.421745472830631) node at (4.8,0.65) {$ \gamma_3 $};
			\end{scriptsize}
			\end{tikzpicture}
		\end{minipage}
		\begin{minipage}{0.3\textwidth}
			\centering
			\begin{tikzpicture}[line cap=round,line join=round,>=triangle 45,x=0.95cm,y=0.95cm, scale=1, every node/.style={scale=1.25}]
			\clip(1.5,0.4) rectangle (6.8,5);
			
			\draw [color=black] (2.54,1.42)-- (6.1455508529661005,1.421745472830631);
			\draw [color=black] (6.1455508529661005,1.421745472830631)-- (4.34126380267011,4.54337136972061);
			\draw [color=black] (4.34126380267011,4.54337136972061)-- (2.54,1.42);
			\draw (2.54,1.42)-- (3.741850425154662,1.42);
			\draw (3.1591384971348377,2.450102890885205)-- (3.741850425154662,1.42);
			\draw (4.94370099114192,1.420581824413233)-- (5.525415285984578,2.4512484223207944);
			\draw (3.759755878631381,3.4915667401792634)-- (4.9431964826727395,3.5031206820547505);
			\begin{scriptsize}
			\draw [fill=black] (2.54,1.42) circle (1pt);
			\draw[color=black] (2.46,1.2) node {$ 222 $};
			
			\draw [fill=black] (6.1455508529661005,1.421745472830631) circle (1pt);
			\draw[color=black] (6.27,1.2) node {$ 333 $};
			
			\draw [fill=black] (4.34126380267011,4.54337136972061) circle (1pt);
			\draw[color=black] (4.35,4.75) node {$ 111 $};
			
			\draw [fill=black] (3.741850425154662,1.42) circle (1pt);
			
			\draw [fill=black] (3.1591384971348377,2.450102890885205) circle (1pt);
			
			\draw [fill=black] (5.525415285984578,2.4512484223207944) circle (1pt);
			
			\draw [fill=black] (4.94370099114192,1.420581824413233) circle (1pt);
			
			\draw [fill=black] (3.759755878631381,3.4915667401792634) circle (1pt);
			
			\draw [fill=black] (4.9431964826727395,3.5031206820547505) circle (1pt);
			
			\draw [fill=black] (2.9406168083848874,1.42) circle (1pt);
			\draw [fill=black] (2.7420495149418573,1.7659332604307565) circle (1pt);
			\draw [fill=black] (2.948461587479838,2.109355085464961) circle (1pt);
			\draw [fill=black] (3.3471207044017284,2.0963282211990313) circle (1pt);
			\draw [fill=black] (3.3412336167697747,1.42) circle (1pt);
			\draw [fill=black] (3.539800910212805,1.7659332604307565) circle (1pt);
			\draw [fill=black] (5.9431664953000976,1.7674829439450774) circle (1pt);
			\draw [fill=black] (5.744934232358041,1.4213575900248316) circle (1pt);
			\draw [fill=black] (5.344317611749981,1.4209697072190324) circle (1pt);
			\draw [fill=black] (5.337775696419512,2.097291911449265) circle (1pt);
			\draw [fill=black] (5.344317611749981,1.4209697072190324) circle (1pt);
			\draw [fill=black] (5.344317611749981,1.4209697072190324) circle (1pt);
			\draw [fill=black] (5.736422013689717,2.1107047569507538) circle (1pt);
			\draw [fill=black] (5.14541547395062,1.7667105498691136) circle (1pt);
			\draw [fill=black] (3.949171980087133,3.844575759426359) circle (1pt);
			\draw [fill=black] (4.160111003663561,3.5060443442959937) circle (1pt);
			\draw [fill=black] (4.143038548477811,4.195232639096031) circle (1pt);
			\draw [fill=black] (4.541908029337653,4.196621140498657) circle (1pt);
			\draw [fill=black] (4.742552256005196,3.8498709112767027) circle (1pt);
			\draw [fill=black] (4.542582972044068,3.5047461857503825) circle (1pt);
			\draw (2.54,1.42)-- (3.741850425154662,1.42);
			\draw (3.1591384971348377,2.450102890885205)-- (3.741850425154662,1.42);
			\draw (4.94370099114192,1.420581824413233)-- (5.525415285984578,2.4512484223207944);
			\draw (3.759755878631381,3.4915667401792634)-- (4.9431964826727395,3.5031206820547505);
			\draw (2.54,1.42)-- (2.9406168083848874,1.42);
			\draw (2.54,1.42)-- (2.7420495149418573,1.7659332604307565);
			\draw (2.7420495149418573,1.7659332604307565)-- (2.9406168083848874,1.42);
			\draw (3.949171980087133,3.844575759426359)-- (4.160111003663561,3.5060443442959937);
			\draw (4.542582972044068,3.5047461857503825)-- (4.742552256005196,3.8498709112767027);
			\draw (4.541908029337653,4.196621140498657)-- (4.143038548477811,4.195232639096031);
			\draw (5.337775696419512,2.097291911449265)-- (5.736422013689717,2.1107047569507538);
			\draw (5.9431664953000976,1.7674829439450774)-- (5.744934232358041,1.4213575900248316);
			\draw (5.344317611749981,1.4209697072190324)-- (5.14541547395062,1.7667105498691136);
			\draw (3.539800910212805,1.7659332604307565)-- (3.3412336167697747,1.42);
			\draw (3.3471207044017284,2.0963282211990313)-- (2.948461587479838,2.109355085464961);

			\draw[darkgray, ->] (4.1,4.25) .. controls (4.,4.4) .. (4.3,4.54337136972061) node at (3.8,4.4) {$ a_3 $};
			
			\draw[darkgray, ->] (4.1,4.18) .. controls (3.1,3.2) .. (2.54,1.467) node at (2.8,3.1) {$ b_3 $};
			
			\draw[darkgray, ->] (4.143038548477811,4.15) .. controls (4.3,0.5) .. (6.1455508529661005,1.421745472830631) node at (4.8,0.7) {$ c_3 $};
			\end{scriptsize}
			\end{tikzpicture} 
		\end{minipage}
		\caption{Probabilities $ a_n, b_n, c_n $ and $ \alpha_n, \beta_n, \gamma_n $ for $ n = 2 $ and $ n = 3 $.} \label{fig:a_n}
	\end{figure}
	For $ n \geq 2 $ define
	\begin{align*} 
	A_n^{(1)} \coloneqq 
	\begin{bmatrix}
	1 & 0 & 0 \\ \alpha_n & \beta_n & \gamma_n \\ \alpha_n & \gamma_n & \beta_n 
	\end{bmatrix}, \;
	A_n^{(2)} \coloneqq 
	\begin{bmatrix}
	\beta_n & \alpha_n & \gamma_n \\ 0 & 1 & 0 \\ \gamma_n & \alpha_n & \beta_n 
	\end{bmatrix}, \;
	A_n^{(3)} \coloneqq  
	\begin{bmatrix}
	\beta_n & \gamma_n & \alpha_n \\ \gamma_n & \beta_n & \alpha_n \\ 0 & 0 & 1
	\end{bmatrix}\!\!.
	\end{align*}
	The matrix $ A_n^{(i)} $ contains exactly the probabilities that the process, starting in one of the three vertices of $ \Delta_i^n $, reaches $ V^n $. More precisely,
	\begin{align*}
	A_n^{(i)} = 
	\begin{bmatrix}
	\bsr(i 1^{n - 1}) \\
	\bsr(i 2^{n - 1}) \\
	\bsr(i 3^{n - 1})
	\end{bmatrix}
	= \begin{bmatrix}
	\rho_1(i 1^{n - 1}) & \rho_2(i 1^{n - 1}) & \rho_3(i 1^{n - 1}) \\
	\rho_1(i 2^{n - 1}) & \rho_2(i 2^{n - 1}) & \rho_3(i 2^{n - 1}) \\
	\rho_1(i 3^{n - 1}) & \rho_2(i 3^{n - 1}) & \rho_3(i 3^{n - 1}) 
	\end{bmatrix}\!\!.
	\end{align*}
	Denote the standard $ i $-th row unit vector of $\R^{3}$ by $ \boldsymbol{e}_i $, for $ i \in \Sigma $.   With the above, we can express the hitting probability vector $ \boldsymbol{\rho}(x) $ as a matrix product, namely,
	\begin{align} \label{matrix_product}
	\boldsymbol{\rho}(x) = \boldsymbol{e}_{i_n} A_2^{(i_{n - 1})} \cdots A_n^{(i_1)}.
	\end{align}
	We investigate the limiting behaviour of these sequences of hitting probabilities to obtain the Martin boundary and the harmonic functions on the boundary. As in \cite{Lau}, this can be done by establishing recursive formulas for these sequences. We note, computing these limits is more involved than in \cite{Lau}. Detailed proofs are given in Section~\ref{sec_hitt_prob}. The main result is the following and is a consequence of Propositions~\ref{prop_limit_b_n} and~\ref{limit_alpha_beta_gamma}. 
	
	\begin{theorem} \label{prop_limits}
		We have that $\displaystyle \lim_{n \to \infty} (\alpha_{n}, \beta_{n}, \gamma_{n}) = (2/5, 2/5, 1/5)$
		and $\displaystyle \lim_{n \to \infty} (a_{n}, b_{n}, c_{n}) = (1, 0, 0)$.
	\end{theorem}
	
	Interestingly, the limits of these sequences are independent of the chosen parameter $ p \in (0,1/2) $ and are equal to the ones obtained in the isotropic case considered in \cite{Lau}.  Loosely speaking, the Martin boundary describes all possible paths of the Markov chain at infinity. Therefore, it is not too surprising that the Martin boundary does not change in a non-isotropic setting, but that the above limits turn out to satisfy the ($1/5$)--($2/5$) law is.  It would be of interest to investigate what happen if one rotates the transition probability, that is, if instead of taking probability $p$ in the directions $[u, v]$ and $[u, w]$, and $q$ in the direction $[u, z]$, as depicted in Figure~\ref{fig:bild-rechts}, one took probability $p$ in the directions $[u, v]$ and $[u, z]$ and $q$ in the direction $[u, w]$.  Simulations indicate that the same results as in Theorem~\ref{prop_limits} may hold.

	\subsection{Limit of the random matrix product} \label{sec_limit_matrix_prod}
	
	Often, as we will see in the proof of Proposition~\ref{prop_limit_Green_function}, it is more convenient to look at the product of the last few matrices of our random matrix product. Define, for $ \mbf x = i_1 i_2 \ldots \in \Sinf $ and $ k \leq n $,
	\begin{align} \label{def_T_n}
	T_n^{\mbf x}  \coloneqq A_2^{(i_n)} \cdots A_{n - k + 1}^{(i_{k + 1})} A_{n - k + 2}^{(i_k)} \cdots A_{n + 1}^{(i_1)} \eqqcolon Q_{n,k}^{\mbf x}  R_{n,k}^{\mbf x} .
	\end{align}
	Here, $ R_{n,k}^{\mbf x}  \coloneqq A_{n - k + 2}^{(i_k)} \cdots A_{n + 1}^{(i_1)} $ and $ Q_{n,k}^{\mbf x}  \coloneqq A_2^{(i_n)} \cdots A_{n - k + 1}^{i_{k + 1}} $.  For $ i \in \Sigma $ let $ A^{(i)} \coloneqq \lim_{n \to \infty} A_n^{(i)} $. By Theorem~\ref{prop_limits}, these limits exist and we have 
	\begin{align} \label{def_matrices} 
	A^{(1)} = \begin{bmatrix}
	1 & 0 & 0 \\
	\nicefrac{2}{5} & \nicefrac{2}{5} & \nicefrac{1}{5} \\
	\nicefrac{2}{5} & \nicefrac{1}{5} & \nicefrac{2}{5}
	\end{bmatrix}, \;\,
	A^{(2)} = \begin{bmatrix}
	\nicefrac{2}{5} & \nicefrac{2}{5} & \nicefrac{1}{5} \\
	0 & 1 & 0 \\
	\nicefrac{1}{5} & \nicefrac{2}{5} & \nicefrac{2}{5}
	\end{bmatrix}, \;\,
	A^{(3)} = \begin{bmatrix}
	\nicefrac{2}{5} & \nicefrac{1}{5} & \nicefrac{2}{5} \\
	\nicefrac{1}{5} & \nicefrac{2}{5} & \nicefrac{2}{5} \\
	0 & 0 & 1
	\end{bmatrix}\!\!.
	\end{align}
	In \cite{Lau} it was shown that the limit of the random matrix product in \eqref{matrix_product} exists and that the limit matrix has identical rows. That is, for $ \mbf x = i_1 i_2 \ldots \in \Sinf $ and $ T_n^{\mbf x} $ as in \eqref{def_T_n}, we have, with \eqref{def_matrices}, that the limit
	\begin{align} \label{thm_T_infty}
	T_{\infty}^{\mbf x} \coloneqq \lim_{n \to \infty} T_n^{\mbf x}  = \lim_{k \to \infty} A^{(i_k)} \cdots A^{(i_1)}
	\end{align}
	exists. 
	To show this, they introduce the concepts of scrambling matrices and the minimum range of a matrix. These are both parameters which measure the difference of the rows of a matrix. 
	
	For the proof of Proposition~\ref{prop_limit_Green_function} concerning the limiting behaviour of the Green function, we require that the hitting probability vector given in \eqref{matrix_product} can be extended to $ \Sinf $. For $ \mbf x = i_1 i_2 \ldots, \mbf y = j_1 j_2 \ldots \in \Sinf $ with $ i_1 \neq j_1 $ we have $ T_{\infty}^{\mbf x} = T_{\infty}^{\mbf y} $ if and only if $ \mbf x = i j^{\infty} $ and $ \mbf y = j i^{\infty} $ with $ i = i_1 $ and $ j = j_1 $. Thus, by the equations given in \eqref{matrix_product} and \eqref{thm_T_infty}, for $ \mbf x \in \Sinf $, the limit
	\begin{align*} 
	\bsr(\mbf x) \coloneqq [\rho_1(\mbf x), \rho_2(\mbf x), \rho_3(\mbf x)] \coloneqq \lim_{n \to \infty} \bsr(\mbf x|_n)
	\end{align*}
	exists and we have, for $ \mbf{x,y} \in \Sinf $, that $ \bsr(\mbf x) = \bsr(\mbf y) $ if and only if $ \mbf x \aq{\pi} \mbf y $. 
	
	\subsection{The Martin space} \label{sec_Martin_metric}
	
	The \textsl{Green function} $G \colon \Sigma^* \times \Sigma^* \to \mathbb{R}^{+}_{0}$ is defined by $ G(x,y) \coloneqq \sum_{n = 0}^{\infty} P^n(x,y) $ for $ x, y \in \Sigma^* $. Observe $ G(x,y) $ is the expected number of visits to $ y $, starting from $ x $.  Further, set  $ \tilde{\rho}_{x,y} \coloneqq \P_x(\, \exists \, n \in \N : X_n = y \, | \, X_0 = x) $.  This value is often referred to as the \textsl{first return time}. Note that $ \tilde{\rho}_{x,y} = \rho_{x,y} $, if $ x \neq y $, and that $ G(x, y) = (\rho_{x,y})/(1 - \tilde{\rho}_{y,y}) $.  The \textsl{Martin kernel} is given by $ K(x,y) \coloneqq \rho_{x,y} / \rho_{\vartheta, y} = G(x,y) / G(\vartheta, y) $.  The latter equation follows from the fact that our Markov chain is transient. Moreover, in general,
	\begin{align} \label{prop_Martin_kernel_always_bounded}
	K(x,y) = \frac{\rho_{x,y}}{\rho_{\vartheta, y}} 
	\leq \frac{\rho_{x,y}}{\rho_{\vartheta, x} \rho_{x,y}} 
	= \frac{1}{\rho_{\vartheta, x}} 
	\eqqcolon C_x.
	\end{align} 
	Thus, for $ x, y \in \Sigma^* $, there exists a constant $ C_x > 0 $, independent of $ y $, such that $ K(x,y) \leq C_x $.
	The function $ \varrho \colon \Sigma^* \times \Sigma^* \to \R_{\geq 0} $ defined by
	\begin{align} \label{eq_def_Martin_metric}
	\varrho(x,y) \coloneqq \lvert r^{|x|} - r^{|y|}\rvert + \sum_{n = 0}^{\infty} r^n \sup_{z \in \Sigma^n} C_z^{-1} \, |K(z,x) - K(z,y)|,
	\end{align}
	where $ r \in (0,1) $ and $ C_z $ is the upper bound for $ K(z, \cdot) $ as given in \eqref{prop_Martin_kernel_always_bounded}, is called a \textsl{Martin metric}.  The Martin metric is indeed a metric on $ \Sigma^* $. A sequence $ (x_n)_{n \in \N} $ in $ \Sigma^* $ is a $ \varrho $-Cauchy sequence if and only if either $ \lim_{n \to \infty} x_n = x $ for some $ x \in \Sigma^* $, or $ |x_n| \to \infty $, where $ |x| $ denotes the length of $ x \in \Sigma^* $, and $ \lim_{n \to \infty} K(z,x_n) $ exists for all $ z \in \Sigma^* $. This property is a characteristic which has to be fulfilled by a Martin metric. Aside from this, there is some freedom in defining a Martin metric. We highlight that the metric considered in \cite{Lau} is different to the one in \eqref{eq_def_Martin_metric}.
	
	Define the equivalence relation $ \aq{\varrho} $ on the set of $ \varrho $-Cauchy sequences by 
	\begin{align*}
	(x_n)_{n \in \N} \aq{\varrho} (y_n)_{n \in \N} 
	\;\; \text{if and only if} \;\;
	\lim_{n \to \infty} \varrho(x_n, y_n) = 0.
	\end{align*}
	The \textsl{Martin space} $ \overline{\Sigma^*} $ of the Markov chain $(X_{n})_{n \in \N}$ is the set of all $ \varrho $-equivalence classes of $ \varrho $-Cauchy sequences in $ \Sigma^* $ and the \textsl{Martin boundary} is defined to be $ \M \coloneqq \partial \overline{\Sigma^*} \coloneqq \overline{\Sigma^*} \setminus \Sigma^*$.
	
	\subsection{The Green function at infinity} \label{The_Green_function}
	
	Here we show that $ \lim_{n \to \infty} G(i^{n-1}, \mbf y|_n) $ exists for all $ \mbf y \in \Sinf $, $ i \in \Sigma $ and $n \in \N$,  from which one may conclude that $ \lim_{n \to \infty} K(x, \mbf y|_n) $ exists for all $ \mbf y \in \Sinf $ and $ x \in \Sigma^* $.  The values of these limits of the Green function are required to extend the Martin metric to $ \Sinf/_{\aq{\pi}} $, which is an important step in proving that the Martin boundary and the Sierpi\'nski gasket are homeomorphic.
	
	\begin{proposition} \label{prop_limit_Green_function}
		Let $t \in \N_{0}$ and $ \mbf x = i^t i_{t + 1} \ldots \in \Sinf $ with $ i_{t + 1} \neq i $ and set $ c \coloneqq 1/(15 \, p) $. For distinct $ j, k \in \Sigma \setminus \{i\} $, we have
		\begin{enumerate}
			\item[(i)] $ \displaystyle\lim_{n \to \infty} G\left(j^{n - 1}, \mbf x|_n\right) = c \, (2 \, \rho_j(\sigma(\mbf x)) + \rho_k(\sigma(\mbf x))) $ and
			\item[(ii)] $ \displaystyle\lim_{n \to \infty} G\left(i^{n - 1}, \mbf x|_n\right) = c \, (5 \, \rho_i(\sigma(\mbf x)) + 2 \, \rho_j(\sigma(\mbf x)) + 2 \, \rho_k(\sigma(\mbf x))).$
		\end{enumerate} 
	\end{proposition}
	
	For the proof of Proposition~\ref{prop_limit_Green_function} we require the following. For $m, n \in \N$ with $ m < n $ and $ x = i_1 \ldots i_n \in \Sigma^{n}$ consider the $ (m,n) $-cell $ \Delta_{i_1 \ldots i_{m - 1}}^n $ that contains $ x $. Let $ \tilde{\Delta}_{i_1 \ldots i_{m - 1}}^n \coloneqq \Delta_{i_1 \ldots i_{m - 1}}^n \setminus \partial \Delta_{i_{1} \ldots i_{m-1}}^{n} $ denote the \textsl{interior} of $ \Delta_{i_1 \ldots i_{m - 1}}^n $. For $ y \in \Sigma^n \setminus \tilde{\Delta}_{i_1 \ldots i_{m - 1}}^n $ define 
	\begin{align} \label{Notation_Green_fct_vector}
	\mbf G(\Delta_{i_1 \ldots i_{m - 1}}^n, y) \coloneqq 
	\begin{bmatrix}
	G(i_1 \ldots i_{m - 1} 1^{n - m + 1}, y) \\
	G(i_1 \ldots i_{m - 1} 2^{n - m + 1}, y) \\
	G(i_1 \ldots i_{m - 1} 3^{n - m + 1}, y) 
	\end{bmatrix}\!\!.
	\end{align}
	Since $ \Delta_{i_1 \ldots i_{n - 1}}^n \subset \Delta_{i_1 \ldots i_{m - 1}}^n $ and since $ y  \notin \tilde{\Delta}_{i_1 \ldots i_{m - 1}}^n $, the chain must pass through one of the outer vertices of $ \Delta_{i_1 \ldots i_{m - 1}}^n $, when starting at a vertex of $ \Delta_{i_1 \ldots i_{n - 1}}^n $,  before it can reach $ y $. By \eqref{matrix_product}, the probability to reach the outer vertices of $ \Delta_{i_1 \ldots i_{m - 1}}^n $, when starting at a vertex of $ \Delta_{i_1 \ldots i_{n - 1}}^n $, is given by 
	\begin{align*}
	Q_{n - 1, m - 1}^x = A_2^{(i_{n-1})} \cdots A_{n - m + 1}^{(i_{m})}.
	\end{align*} 
	Therefore, we have
	\begin{align} \label{Notation_Green_fct}
	\mbf G(\Delta_{i_1 \ldots i_{n - 1}}^n, y) 
	= Q_{n - 1, m - 1}^x \mbf G(\Delta_{i_1 \ldots i_{m - 1}}^n, y).
	\end{align}
	Indeed, here we have used the fact that 
	\begin{align}\label{thm_cut_point}
	\rho_{x,y} = \sum_{z \in \Sigma^*} P(x,z) \, \rho_{z,y} ,
	\end{align}
	for distinct $x, y \in \Sigma^*$, and that this equality generalises to subsets of $ \Sigma^* \setminus \{x\} $:  for $ A \subset \Sigma^* \setminus \{x\} $ such that each path from $ x $ to $ y $ contains a state of $ A $, we have
	\begin{align} \label{thm_cut_set}
	\rho_{x,y}
	= \sum_{a \in A} \P ( 
	\, \exists \; n \in \N \colon X_{n} = a \; \text{and} \; X_{m} \not\in A \; \forall \; m < n  
	\mid X_0 = x ) \, \rho_{a,y}.
	\end{align}
	We refer the reader to \cite{MR2548569} for a proof of \eqref{thm_cut_point} and \eqref{thm_cut_set}.
	
	\begin{proof}[Proof of Proposition~\ref{prop_limit_Green_function}]
		Let $t, n \in \N$ with $ t < n $ and let $ \mbf x = i^t i_{t + 1} \ldots \in \Sinf $ with $ i_{t + 1} \neq i $.
		
		We first prove statement (i). Note that $ j^n \notin \Delta_{i}^n $ for $ j \neq i $. Thus, with \eqref{Notation_Green_fct_vector} and \eqref{Notation_Green_fct} it follows that
		\begin{align} \label{proofeq1_limit_Green_fct}
		G\left(\mbf x|_n, j^n \right) = \boldsymbol e_{i_n} \mbf G\left(\Delta_{i^t i^{t + 1} i_{t + 2} \ldots i_{n - 1}}^n, j^n \right)
		= \boldsymbol e_{i_n} Q_{n - 1, 1}^{\mbf x} \mbf G\left(\Delta_i^n, j^n \right). 
		\end{align}
		By \eqref{matrix_product} and \eqref{thm_T_infty}, we have that
		\begin{align*}
		\lim_{n \to \infty} \boldsymbol e_{i_n} Q_{n - 1, 1}^{\mbf x}  
		= \lim_{n \to \infty} \boldsymbol e_{i_n} \A{2}{i_{n - 1}} \cdots \A{n - 1}{i_2} 
		= \bsr(\sigma(\mbf x)).
		\end{align*}
		Observe, for $ \mbf x \in \Sinf $ with $ \mbf x|_n \notin V^n $, that $ 3 p \, G\left(i^{n - 1}, \mbf x|_n \right) = G\left(\mbf x|_n, i^n \right) $. 
		
		For $ x \in \tilde{\Sigma}^n $, recall that $ G(x, j^n) = (\rho_{x,j^n})/(1 - \tilde{\rho}_{j^n,j^n}) = \rho_j(x) $. Therefore, $ G(ik^{n - 1}, j^n) = \gamma_n $ and $ G(ij^{n - 1}, j^n) = \beta_n $.  This in tandem with Proposition~\ref{prop_limits} and~\eqref{proofeq1_limit_Green_fct} implies that
		\begin{align*}
		\lim_{n \to \infty} G\left(j^{n - 1}, \mbf x|_n\right)
		&= \frac{1}{3p} \lim_{n \to \infty} G\left(\mbf x|_n, j^{n}\right)\\
		&=  \frac{1}{3p} \bsr(\sigma(\mbf x)) \lim_{n \to \infty} \mbf G\left(\Delta_i^n, j^{n}\right) \\
		&= \frac{1}{3p} \left( \rho_j(\sigma(\mbf x)) \lim_{n \to \infty} G\left(ij^{n - 1}, j^{n}\right) + \rho_k(\sigma(\mbf x)) \lim_{n \to \infty} G\left(ik^{n - 1}, j^{n}\right) \right) \\
		&= \frac{1}{3p} \left( \rho_j(\sigma(\mbf x)) \lim_{n \to \infty} \beta_n + \rho_k(\sigma(\mbf x)) \lim_{n \to \infty} \gamma_n \right)\\
		&= \frac{1}{3p} \frac{1}{5} \left( \, 2 \, \rho_j(\sigma(\mbf x)) + \rho_k(\sigma(\mbf x)) \, \right).
		\end{align*}
		This completes the proof of statement (i), and so, we turn our attention to the proof of statement (ii).  Without loss of generality assume $ i = 1 $ and $ i_{t + 1} = 2 $. Since $ 1^n \notin \Delta_{1^t 2}^n $, by \eqref{Notation_Green_fct}, we have that 
		$ \mbf G(\Delta_{1^t 2 i_{t+2} \ldots i_{n - 1}}, 1^n) = Q_{n - 1, t + 1}^{\mbf x} \mbf G(\Delta_{1^t 2}^n, 1^n)$.  Further, by \eqref{matrix_product}, we have 
		\begin{align*}
		\mbf G(\Delta^n_{1^t 2}, 1^n) 
		&= \begin{bmatrix}
		\boldsymbol{\rho}(1^t 2 1^{n - t - 1}) \\
		\boldsymbol{\rho}(1^t 2 2^{n - t - 1}) \\
		\boldsymbol{\rho}(1^t 2 3^{n - t - 1})
		\end{bmatrix}
		\boldsymbol{e}_1^T
		= A_{n - t}^{(2)} A_{n - t + 1}^{(1)} \cdots A_{n}^{(1)}  \boldsymbol{e}_1^T
		= A_{n - t}^{(2)} A_{n - t + 1}^{(1)} \cdots A_{n-1}^{(1)} \mbf G(\Delta_1^n, 1^n),
		\end{align*}
		and so 
		\begin{align*}
		\mbf G(\Delta_{1^t 2 i_{t + 2} \ldots i_{n - 1}}, 1^n) 
		= Q_{n - 1, t + 1}^{\mbf x}  A_{n - t}^{(2)} A_{n - t + 1}^{(1)} \cdots A_{n-1}^{(1)} \mbf G(\Delta_1^n, 1^n)  
		= Q_{n - 1, 1}^{\mbf x}  \mbf G(\Delta_1^n, 1^n).	
		\end{align*}
		The remainder of the proof follows analogously to that of statement (i), from which one obtains, 
		\[
		\lim_{n \to \infty} G\left(1^{n - 1}, \mbf x|_n\right)
		= \frac{1}{3p} \frac{1}{5} \left( \, 5 \, \rho_1(\sigma(\mbf x)) + 2 \, \rho_2(\sigma(\mbf x)) + 2 \, \rho_3(\sigma(\mbf x)) \, \right).
		\qedhere
		\]
	\end{proof}
	
	Above, we were able to decompose $ \mbf G(\Delta_{i^q j i_{q + 2} \ldots i_{n - 1}^n}, i^n) = Q_{n - 1, 1}^{\mbf x} \mbf G(\Delta_i^n, i^n) $, even if $ i^n \in \Delta_{i^l}^n $; for distinct $ i, j \in \Sigma $ and $ l \leq q $. Notice this is only possible since $ i^n $ is one of the three outer vertices of each $ n $-cell $ \Delta_{i^l}^n $. This provides a simpler and alternative proof to that given in \cite{Lau}. In general, for a state $ y \in \tilde{\Delta}_{i^l}^n $, the aforementioned decomposition is not possible, as the chain does not need to go via one of the outer vertices of the cell to reach $ y $. Note that the constant $ c $, which we obtain in the limits, is different from the one(s) in \cite{Lau}.
	
	With the above at hand, in particular Proposition~\ref{prop_limit_Green_function}, we may extend the Martin kernel $ K(x, \cdot \, ) $ continuously to $ \Sinf $.  For $ z \in \Sigma^m $ and $ y \in \tilde{\Sigma}^n $, for some $ n \geq m + 1 $,
	\begin{align} \label{eq:martin_eq}
	K(z,y) = \frac{\sum_{i = 1}^3 \rho_{z, i^{n - 1}} G\left(i^{n - 1}, y\right) }{\frac{1}{3} \sum_{i = 1}^3 G\left(i^{n - 1}, y\right)}.
	\end{align}
	
	\begin{proposition} \label{prop_K(x,.)ex_ist_stetig}
		For $ x \in \Sigma^* $ and $ \mbf y \in \Sinf $ the following limit exists.
		\begin{align*}
		K(x, \mbf y) \coloneqq \lim_{n \to \infty} K(x, \mbf y|_n)
		\end{align*}
	\end{proposition} 
	
	Note, in \cite{Lau}, a slightly different decomposition to that given in \eqref{eq:martin_eq} is used.
	
	\subsection{The Martin boundary} \label{The_Martin_metric_and_Martin_space}
	
	Proposition \ref{prop_K(x,.)ex_ist_stetig} implies, for $ \mbf x,  \mbf y  \in \Sinf $, that $ ( \mbf x|_n )_{n \in \N} $ and $ ( \mbf y|_n )_{n \in \N} $ are $ \varrho $-Cauchy sequences. Thus, $ \varrho(\mbf x|_n, \mbf y|_n) $ is a Cauchy sequence and $ \varrho(\mbf{x,y}) \coloneqq \lim_{n \to \infty} \varrho(\mbf x|_n, \mbf y|_n) $ is well defined.  Moreover, $ \mbf x \aq\pi \mbf y $ implies that $ \lim_{n \to \infty} K(z,\mbf x|_n ) = \lim_{n \to \infty} K(z,\mbf y|_n ) $ for all $ z \in \Sigma^* $. Hence, $ \Sinf/_{\aq\pi} $ is a subset of the Martin boundary and $ \varrho $ is well defined on $ \Sinf/_{\aq\pi} $.
	
	Define $ d \colon \Sinf \times \Sinf \to \R $ by $ d(\mbf{x,y}) \coloneqq 2^{-\max\{n \colon \mbf x|_n = \mbf y|_n\} } $. 
	It is well known that the function $ d $ is a metric on $ \Sinf $ and that $ (\Sinf, d) $ is a compact metric space, see for instance \cite{KF:1990}, \cite{Kig} and \cite{Str}. Moreover, $ \pi $ is continuous with respect to $ d $ and surjective.  Let $ \mathcal{Q} $ be the quotient topology induced from the standard metric $ d $ on the space $ \qur{\Sinf}{\pi} $. To prove that the Martin boundary $ \M $ is homeomorphic to the Sierpi\'nski gasket $\mathcal{K}$ it is necessary to show the homeomorphisms $(\mathcal{K}, |\cdot|) \cong (\qur{\Sinf}{\pi}, \mathcal{Q}) \cong (\qur{\Sinf}{\pi}, \varrho) \cong (\M, \varrho)$. Here, an important step is to show that the Martin metric $ \varrho $ defines a metric on $ \Sinf/_{\aq\pi} $. The main difficulty lies in showing that $ \mbf x \nsim_{\pi} \mbf y $ implies $ \varrho(\mbf{x,y}) > 0 $. The following theorem can be shown as in \cite{Lau}.
	
	\begin{theorem} \label{thm_M_homeo_K}
		The Martin boundary $ \M $ of $ (X_n)_{n \in \N_0} $ is homeomorphic to the Sierpi\'nski gasket $\mathcal{K}$.
	\end{theorem}
	
	\subsection{Harmonic functions} \label{sec_harm_fct}

	A function $ u \colon \Sigma^* \to \R $ is called \textsl{$ P $-harmonic function} if for all $ x \in \Sigma^* $ we have $ P u(x) \coloneqq \sum_{y \in \Sigma^*} P(x,y) u(y) = u(x) $.
	Such a function $ u $ is called \textsl{minimal} if $u(\vartheta) = 1$ and if $ 0 \leq v(x) \leq u(x) $ for all $ x \in \Sigma^* $ and some $P$-harmonic function $ v $, implies that $ v = c \, u $ for some constant $ c \geq 0 $. 
	The \textsl{minimal Martin boundary or exit space} $ \mathcal{M}_{\text{min}} $ of a Markov chain is defined to be 
	\begin{align*}
	\M_{\text{min}} \coloneqq \{ \, \xi \in \mathcal{M} \; | \; K(\cdot, \xi) \text{ is a minimal harmonic function} \, \}.
	\end{align*}
	The minimal Martin boundary is a Borel subset of $ \M $, see for instance \cite{MR2548569}. In many cases, the minimal Martin boundary equals the Martin boundary. However, this is not the case in our setting, and in fact, we have the following result, which is a consequence of a general result that states that a Markov chain converges almost surely to the minimal Martin boundary, see for instance \cite{Dyn,KKS,MR2548569}.
	
	\begin{theorem} \label{thm_min_Martin_boundary}
		The minimal Martin boundary $\M_{\text{min}}$ of $ (X_n)_{n \in \N_0} $ is homeomorphic to $ \{1^{\infty}, 2^{\infty}, 3^{\infty} \} $.
	\end{theorem}
	
	A non-negative $ P $-harmonic function $ u $ has a unique integral representation over the minimal Martin boundary, namely, there exists a measure $\nu$ supported on $\M_{\textrm{min}}$ such that 
	\begin{align*}
	u(x) = \int K(x, \xi) \dx\nu(\xi).
	\end{align*}
	This together with Theorem~\ref{thm_min_Martin_boundary} implies that each non-negative $ P $-harmonic function is a linear combination of the $ P $-harmonic functions $ h_i(x) \coloneqq K(x, i^{\infty}) $, where $ i \in \Sigma $.
	The extension $ h_i(\mbf x) \coloneqq \lim_{n \to \infty} h_i(\mbf x|_n) $ is well defined and continuous on $ \M $. 
	The proof of this result follows in the same manner as in \cite{Lau}.
	
	\definecolor{ffzzqq}{rgb}{1,0.6,0}
	\begin{figure}
		\begin{center}
			\resizebox{12.5cm}{!}{
				\begin{tikzpicture}[line cap=round,line join=round,>=triangle 45,x=0.8cm,y=0.8cm]
				\clip(-2.5,-1) rectangle (14.5,4.5);
				\draw [line width=1.2pt] (-2,0)-- (2,0);
				\draw [line width=1.2pt] (2,0)-- (0,3.46);
				\draw [line width=1.2pt] (0,3.46)-- (-2,0);
				\draw [line width=1.2pt] (4,0)-- (8,0);
				\draw [line width=1.2pt] (8,0)-- (6,3.46);
				\draw [line width=1.2pt] (6,3.46)-- (4,0);
				\draw [line width=1.2pt] (10,0)-- (14,0);
				\draw [line width=1.2pt] (14,0)-- (12,3.46);
				\draw [line width=1.2pt] (12,3.46)-- (10,0);
				\draw [line width=1.2pt] (5,1.73)-- (6,0);
				\draw [line width=1.2pt] (6,0)-- (7,1.73);
				\draw [line width=1.2pt] (7,1.73)-- (5,1.73);
				\draw [line width=1.2pt] (11,1.73)-- (12,0);
				\draw [line width=1.2pt] (12,0)-- (13,1.73);
				\draw [line width=1.2pt] (13,1.73)-- (11,1.73);
				\draw [line width=1.2pt] (11.5,2.6)-- (12,1.73);
				\draw [line width=1.2pt] (12,1.73)-- (12.5,2.6);
				\draw [line width=1.2pt] (12.5,2.6)-- (11.5,2.6);
				\draw [line width=1.2pt] (10.5,0.87)-- (11,0);
				\draw [line width=1.2pt] (11.5,0.87)-- (10.5,0.87);
				\draw [line width=1.2pt] (11.5,0.87)-- (11,0);
				\draw [line width=1.2pt] (12.5,0.87)-- (13,0);
				\draw [line width=1.2pt] (13,0)-- (13.5,0.87);
				\draw [line width=1.2pt] (13.5,0.87)-- (12.5,0.87);
				\draw (-0.25,-0.4) node[anchor=north west] {$G_0$};
				\draw (5.75,-0.4) node[anchor=north west] {$G_1$};
				\draw (11.75,-0.4) node[anchor=north west] {$ G_2 $};
				\begin{scriptsize}
				\fill [color=black] (-2,0) circle (2.5pt);
				\fill [color=black] (2,0) circle (2.5pt);
				\fill [color=black] (0,3.46) circle (2.5pt);
				\fill [color=black] (4,0) circle (2.5pt);
				\fill [color=black] (8,0) circle (2.5pt);
				\fill [color=black] (10,0) circle (2.5pt);
				\fill [color=black] (14,0) circle (2.5pt);
				\fill [color=black] (6,3.46) circle (2.5pt);
				\fill [color=black] (12,3.46) circle (2.5pt);
				\fill [color=black] (5,1.73) circle (2.5pt);
				\fill [color=black] (7,1.73) circle (2.5pt);
				\fill [color=black] (6,0) circle (2.5pt);
				\fill [color=black] (11,1.73) circle (2.5pt);
				\fill [color=black] (13,1.73) circle (2.5pt);
				\fill [color=black] (12,0) circle (2.5pt);
				\fill [color=black] (11.5,2.6) circle (2.5pt);
				\fill [color=black] (10.5,0.87) circle (2.5pt);
				\fill [color=black] (11,0) circle (2.5pt);
				\fill [color=black] (13,0) circle (2.5pt);
				\fill [color=black] (13.5,0.87) circle (2.5pt);
				\fill [color=black] (12.5,2.6) circle (2.5pt);
				\fill [color=black] (12,1.73) circle (2.5pt);
				\fill [color=black] (12.5,0.87) circle (2.5pt);
				\fill [color=black] (11.5,0.87) circle (2.5pt);
				\end{scriptsize}
				\end{tikzpicture}
			}
			\caption{ Graph approximations $G_{n}$ of the Sierpi\'nki gasket, for $n = 0, 1$ and $2$.} \label{Fig_Graphen_Strichartz}
		\end{center}
	\end{figure}
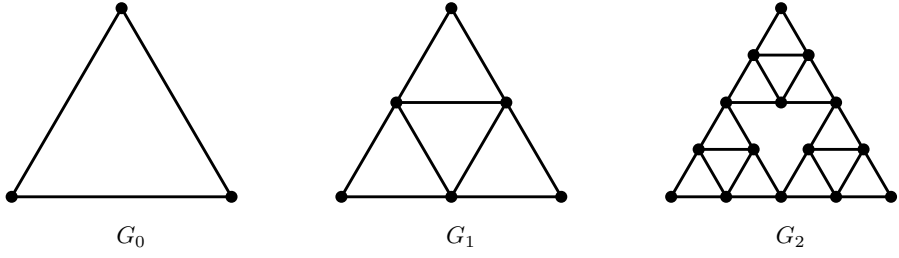
	As shown in \cite{Kig,Str}, the Sierpi\'nki gasket can be approximated by the graphs $ G_n $, $ n \in \N_0 $, see Figure~\ref{Fig_Graphen_Strichartz}. The vertices of these graphs are $ S_{\omega}(q_i) $ for $ i \in \Sigma $ and $ \omega \in \Sigma^n $.  In \cite{Kig,Str},  it was shown that the value $ h(x) $ of a harmonic function $ h $ on the Sierpi\'nki gasket $\mathcal{K}$ at $ x \in G_n $ is a weighted sum over the values $ h(S_{\omega}(q_i)) $, $ i \in \Sigma $, for an appropriate $ \omega \in \Sigma^{n - 1} $.  Specifically, $\omega$ is so that $ S_{\omega}(\mathcal{K}) $ is the unique $ (n - 1) $-cell containing $ x $. This is the so called \textsl{$ 1/5 $--$2/5 $-rule}. Note that $ \pi(\omega j^{\infty}) = S_{\omega}(q_j) $ and consider $ \omega j^{\infty} $ as corresponding vertex to $ S_{\omega}(q_j) $ in $ \Sinf $. We have the following equivalent statement of the $ 1/5 $--$2/5 $-rule. 
	For $ j, k, l \in \Sigma $ pairwise distinct, $ \omega \in \Sigma^{m - 1} $ and $ i \in \Sigma $,
	\begin{align} \label{prop_1_5_2_5_rule}
	h_i(\omega j k^{\infty}) = \frac{2}{5} \, h_i(\omega j^{\infty}) + \frac{2}{5} \, h_i(\omega k^{\infty}) + \frac{1}{5} \, h_i(\omega l^{\infty}).
	\end{align}
	
	This can be shown as in \cite{Lau}. In fact, the weights in \eqref{prop_1_5_2_5_rule} come from the limits of the sequences of hitting probabilities, namely that $ \lim_{n \to \infty} \alpha_n = \lim_{n \to \infty} \beta_n = 2/5 $ and $ \lim_{n \to \infty} \gamma_n = 1/5 $, see Theorem~\ref{prop_limits}.  The equality given in \eqref{prop_1_5_2_5_rule} is equivalent to the property that the value of a harmonic function $ h $ of an inner vertex of $ G_n $ is the average over the values $ h(y) $ for all four neighbouring vertices $ y $ of $ x $ in $ G_n $. This harmonic graph property uniquely determines harmonic functions on the Sierpi\'nski gasket, see for instance \cite{Kig,Str}. 
	
	\begin{theorem}\label{thm:harmonic_functions}
		The $ P $-harmonic functions on the Martin boundary coincide with the canonical harmonic functions of \cite{Kig,Str}, and hence the space of $P$-harmonic functions on $\mathcal{K}$ is \mbox{three-dimensional}. 
	\end{theorem}
	
	Notice that the harmonic functions, although they do not vary with the parameter $ p $ on the Martin boundary, for different values of $ p $, they certainly differ on the state space $ \Sigma_* $. 
	
	\section{Basic hitting probabilities} \label{sec_hitt_prob}
	
	In this section we prove Theorem~\ref{prop_limits}. Indeed, this result follows from Propositions~\ref{prop_limit_b_n} and~\ref{limit_alpha_beta_gamma}. This requires several technical lemmata and the following recursive formulas for the hitting probabilities $ \alpha_n, \beta_n, \gamma_n $ and $ a_n, b_n, c_n $. In the sequel, as above, let $ p \in (0,1/2) $ be fixed and recall that $q \coloneqq 1 - 2p$. 
	
	By \eqref{thm_cut_point} and by definition of $ a_2 $, $ b_2 $ and $ c_2 $, we have that $ a_2 = p + p \, b_2 + q \, a_2 $, $ b_2 = p \, a_2 + q \, c_2 $ and $ c_2 = p \, c_2 + q \, b_2 $.  Solving these linear equations yields
	\begin{align} \label{eq_startwerte}
	a_2 = \frac{3 - 4p}{5 - 7p} = \alpha_2, \hspace{1em} 
	b_2 = \frac{1 - p}{5 - 7p}= \beta_2, \hspace{1em}
	c_2 = \frac{1 - 2p}{5 - 7 p} = \gamma_2.
	\end{align}
	Similarly, for $ n > 2 $, one may use symmetry to obtain 
	\begin{align*}
	\alpha_{n + 1} = p \; \beta_{n + 1} + p \; \rho_{12^{n - 1}1, 1^{n + 1}} + q \; \rho_{12^{n - 1}3, 1^{n + 1}}.
	\end{align*}
	This in combination with \eqref{thm_cut_set}, where $ A = \{1^{n+1}, 12^n, 13^n \} $, yields that
	\begin{align*}
	\alpha_{n + 1} = p \; \beta_{n + 1} + p \; ( a_n \alpha_{n + 1} + c_n \alpha_{n + 1} + b_n ) 
	+ q \; ( a_n \alpha_{n + 1} + b_n \alpha_{n + 1} + c_n ).
	\end{align*}
	\noindent With similar arguments for $ \beta_n, \gamma_n $ and $ a_n, b_n, c_n $ it follows that
	\begin{align}
	&\begin{aligned} \label{dependencies}
	&\alpha_{n + 1} ( 1 - a_n (1 - 2p) - b_n (1 - 3p) - p ) = p \beta_{n + 1} + p b_n + c_n (1 - 2p),\\
	&\beta_{n + 1} ( 1 - a_n (1 - p) ) = p \alpha_{n + 1} + \gamma_{n + 1} ( p c_n + b_n (1 - 2p) ),\\ 
	&\gamma_{n + 1} (1 - p) (1 - a_n) = \beta_{n + 1} \left( p c_n + b_n (1 - 2p) \right),\\
	&a_{n + 1} = a_n + b_n \alpha_{n + 1} + c_n \alpha_{n + 1},\\
	&b_{n + 1} = b_n \beta_{n + 1} + c_n \gamma_{n + 1},\\
	&c_{n + 1} = b_n \gamma_{n + 1} + c_n \beta_{n + 1}.
	\end{aligned}
	\intertext{Rearranging these equations we obtain the following recursive formulas}
	&\begin{aligned} \label{rec_formulas}
	&\alpha_{n + 1} = ( (b_n + c_n) (1 - p) p + c_n^2 (1 - 2 p) + b_n^2 p (2 - 3 p) + b_n c_n (2 - 6 p(1 - p)) ) / d_n,\\
	&\beta_{n + 1} = (b_n + c_n) (1 - p) p / d_n,\\
	&\gamma_{n + 1} = p (b_n (1 - 2 p) + c_n p) / d_n,\\
	&a_{n + 1} = ( c_n (2 - p) p + c_n^2 (1 - 3 p) + b_n p (3 - 4 p) + b_n c_n (2 - 9 p + 9 p^2 ) ) / d_n,\\
	&b_{n + 1} = p  (b_n c_n (2 - 3 p) + b_n^2 (1 - p) + c_n^2 p  ) / d_n,\\
	&c_{n + 1} = p  (b_n c_n + c_n^2 (1 - p) + b_n^2 (1 - 2 p)  ) / d_n,
	\end{aligned}
	\intertext{where the denominator $ d_n $ is given by}
	&d_n \coloneqq {c_n (2 - p) p + c_n^2 (1 - 2 p) + b_n^2 p (2 - 3 p) + b_n p (3 - 4 p) + b_n c_n (2 - 6 p + 6 p^2 )}. \notag
	\end{align}
	
	Next we identify the limits of these sequences. 
	
	\begin{lemma} \label{lemma_b_n_groesser_c_n}
		$ b_n \geq c_n $ for all $ n \in \N $ with $ n \geq 2 $.
	\end{lemma}
	
	\begin{proof}
		With \eqref{eq_startwerte} we have $ b_2 \geq c_2 $. Assume $ b_n \geq c_n $ for a $ n \geq 2 $. 
		By \eqref{rec_formulas}, $ b_{n + 1} \geq c_{n + 1} $ is  equivalent to
		\begin{align*}
		b_n c_n (2 - 3p) + b_n^2 (1 - p) + c_n^2 p 
		&\geq  b_n c_n + c_n^2 (1 - p) + b_n^2 (1 - 2p).
		\end{align*}
		Thus, it is sufficient to show that $ b_n^2 p + c_n^2 (2p - 1) + b_n c_n (1 - 3 p ) \geq 0 $.
		Using the assumption that $ b_n \geq c_n $ and the fact that $ 1 - 2p \geq 0 $, we obtain
		\begin{align*}
		b_n^2 p + c_n^2 (2p - 1) + b_n c_n (1 - 3 p) 
		&= b_n^2 p + c_n^2 (2p - 1) + b_n c_n (1 - 2 p) - b_n c_n p \\
		&\geq b_n^2 p + c_n^2 (2p - 1) + c_n^2 (1 - 2 p) - b_n^2 p = 0.
		\qedhere
		\end{align*}
	\end{proof}
	
	\begin{lemma} \label{lemma_b_n_decr}
		The sequence $ (b_n)_{n \geq 2} $ is monotonically decreasing.
	\end{lemma}
	
	\begin{proof}
		By \eqref{rec_formulas}, for $ n \geq 2 $, it is sufficient to show
		\begin{align*}
		b_n c_n p (2 - p) + b_n c_n^2 (1 - 2 p)&+ b_n^3 p (2 - 3 p) + b_n^2 p (3 - 4 p) + b_n^2 c_n (2 - 6 p + 6 p^2 )\\
		&\geq b_n c_n p (2 - 3p) + b_n^2 p (1 - p) + c_n^2 p^2 ,
		\end{align*}
		which is equivalent to 
		\begin{align} \label{label}
		2 b_n c_n p^2 \!- c_n^2 p^2 + b_n c_n^2 (1 - 2p) + b_n^2 p (b_n + 1)(2 - 3p) + b_n^2 c_n (2 - 6 p + 6 p^2 )
		\end{align}
		being non-negative.  Lemma~\ref{lemma_b_n_groesser_c_n} implies $ 2 b_n c_n p^2 - c_n^2 p^2 \geq 2 c_n^2 p^2 - c_n^2 p^2 = c_n^2 p^2 \geq 0 $, and since $ p \in (0, 1/2) $, it follows that $ (1 - 2p) \geq 0 $ and $ (2 - 3p) \geq 0 $. Further, $ f(p) \coloneqq 2 - 6 p + 6 p^2 > 0 $ attains its minimum at $ p = 1/2 $ and $ f(1/2) = 1/2 $. Thus, \eqref{label} is valid.
	\end{proof}
	
	We aim to show that $ \lim_{n \to \infty} b_n = 0 $. For this, we first obtain an upper bound for the sequence $ (\beta_n)_{n \geq 2} $ and a lower bound for $ (\alpha_n)_{n \geq 2} $.
	
	\begin{lemma} \label{beta_kleiner_2_5}
		For $ p \leq 1/3 $ and $n \in \N$ with $ n \geq 2 $, we have $ \beta_n \leq 2/5 $.
	\end{lemma}
	
	\begin{proof}
		We have $ \beta_2 = (1 - p)/(5 - 7p) \leq 2/5 $ if and only if $ p \leq 5/9 $, which holds since $ p \leq 1/3 $.
		Let $ n \in \N $ with $ n > 2 $. We claim $ \beta_{n + 1} \leq 2/5 $, which is equivalent to 
		\begin{align*}
		\frac{(b_n + c_n) (1 - p) p}{c_n (2 - p) p + c_n^2 (1 - 2 p) + b_n^2 p (2 - 3 p) + b_n p (3 - 4 p) + b_n c_n (2 - 6 p + 6 p^2)} \leq \frac{2}{5}.
		\end{align*}
		This in turn is equivalent to 
		\begin{align*}
		c_n p (3p - 1) + c_n^2 2(1 - 2 p) + b_n^2 2 p (2 - 3 p) + b_n p (1 - 3p) + b_n c_n 2 (2 - 6 p + 6 p^2 ) \!\geq\! 0. 
		\end{align*} 
		Since $ p \leq 1/3 $, the terms on the left hand side are all positive except for $ c_n \, p \, (3p - 1) $. Combining this with Lemma~\ref{lemma_b_n_groesser_c_n}, we have $ c_n \, p \, (3p - 1) + b_n \, p \, (1 - 3p) \geq 0 $, yielding the result. 
	\end{proof}
	
	\begin{lemma} \label{alpha_gr_2_5}
		For $ p \geq 1/3 $ and all $ n \geq 2 $ we have $ \alpha_n \geq 2/5 $.
	\end{lemma}
	
	\begin{proof} 
		We have $ \alpha_2 = (3 - 4p)/(5 - 7p) \geq 2/5 $ if and only if $ p \leq 5/6 $, which holds since $ p \leq 1/2 $. For $ n \in \N $ with $ n > 2 $, we claim that $ \alpha_{n + 1} \geq 2/5 $. By \eqref{rec_formulas}, this is equivalent to 
		\begin{align*}
		(b_n - c_n) \, p \, (3p - 1) + c_n^2 \, 3 \, (1 - 2p) + b_n^2 3 \, p \, (2 - 3p) + b_n c_n 6 \, ( 1 - 3 p + 3 p^2 ) \geq 0.
		\end{align*}
		This inequality holds since, by Lemma~\ref{lemma_b_n_groesser_c_n}, we have $ b_n \geq c_n $, and since $ (3 p - 1) \geq 0 $ for $ p \geq 1/3 $. 
	\end{proof}
	
	\begin{lemma} \label{lemma_gamma_kleiner_beta}
		For $ n \in \N $ with $ n \geq 2 $, we have that $ \beta_n \geq \gamma_n $.
	\end{lemma}
	
	\begin{proof}
		By \eqref{eq_startwerte}, we have that $ \beta_2 = b_2 \geq c_2 = \gamma_2 $ and, for $ n > 2 $, by \eqref{dependencies},
		\begin{align*}
		\gamma_{n + 1} = \beta_{n + 1} \frac{b_n (1 - 2p) + c_n p}{(1 - p)(b_n + c_n)}.
		\end{align*}
		Since $ 0 \leq b_n p + c_n (1 - 2p) $, which is equivalent to $ b_n (1 - 2p) + c_n p \leq (1 - p)(b_n + c_n) $, it follows that  $\beta_{n + 1} \geq \gamma_{n + 1}$.
	\end{proof}
	
	\begin{proposition} \label{prop_limit_b_n}
		$\displaystyle  \lim_{n \to \infty} b_n = 0 $
	\end{proposition}
	
	\begin{proof}
		Let $ p \leq 1/3 $. Lemmata \ref{lemma_b_n_groesser_c_n}, \ref{beta_kleiner_2_5} and \ref{lemma_gamma_kleiner_beta} together with \eqref{dependencies} imply, for all $n \in \N$, 
		\begin{align*}
		b_{n + 1} = b_n \beta_{n + 1} + c_n \gamma_{n + 1} 
		\leq 2 \, b_n \beta_{n + 1} \leq b_n (4/5).	
		\end{align*}
		Thus, $ b_n \leq b_2 \left( 4/5 \right)^{n - 2} $, and so, $ \lim_{n \to \infty} b_n = 0 $.
		
		Conversely, for $ p > 1/3 $, by Lemma~\ref{alpha_gr_2_5}, we have $ 0 \leq \gamma_n + \beta_n = 1 - \alpha_n \leq 3/5 $. This in tandem with Lemma~\ref{lemma_b_n_groesser_c_n} and~\eqref{dependencies} yields that $ b_{n + 1} \leq b_n (3/5) $. 
		Thus, $ b_n \leq b_2 \left( 3/5 \right)^{n - 2} $, and so, $ \lim_{n \to \infty} b_n = 0 $.
	\end{proof}
	
	\begin{corollary} \label{cor_lim_a_n_lim_c_n}
		We have that $ \displaystyle \lim_{n \to \infty} c_n = 0 $ and $ \displaystyle \lim_{n \to \infty} a_n = 1 $.
	\end{corollary}
	
	\begin{corollary} \label{bound_b_n_plus_c_n}
		For $ n \geq 2 $, if $ p > 1/3 $, then $ b_n + c_n \leq \left( 3/5\right)^{n - 2} $, and if $ p \leq 1/3 $, then $ b_n + c_n \leq \left( 4/5 \right)^{n - 2} $. 
	\end{corollary}
	
	\begin{proof}
		Let $ p > 1/3 $. As in the proof of Proposition~\ref{prop_limit_b_n}, we have, by Lemma~\ref{alpha_gr_2_5} and~\eqref{dependencies}, that
		\begin{align*}
		b_{n + 1} + c_{n + 1} &= (b_n + c_n) (\beta_{n + 1} + \gamma_{n + 1}) 
		= \prod_{k = 2}^{n + 1} (\beta_{k} + \gamma_{k})
		= \prod_{k = 2}^{n + 1} (1 - \alpha_k)
		\leq \left( \frac{3}{5} \right)^{n - 1}. 
		\end{align*}
		Similarly, for $ p \leq 1/3 $, by the Lemmata \ref{beta_kleiner_2_5} and \ref{lemma_gamma_kleiner_beta} and \eqref{dependencies}, we have that 
		\[
		b_{n + 1} + c_{n + 1} = \prod_{k = 2}^{n + 1} (\beta_{k} + \gamma_{k}) 
		\leq 2 \, \prod_{k = 2}^{n + 1} \beta_{k}
		\leq \left( \frac{4}{5} \right)^{n - 1}.
		\qedhere
		\]
	\end{proof}
	
	We turn our attention to finding the limits of the sequences $ (\alpha_n)_{n \in \N}, (\beta_n)_{n \in \N} $ and $ (\gamma_n)_{n \in \N}  $. First we observe that $ \alpha_n $ and $ \beta_n $ have to converge to the same value, provided the limits exist. 
	
	\begin{lemma} \label{gleicher_GW_beta_alpha}
		If the limit $ \displaystyle \lim_{n \to \infty} \alpha_n $ exists, we have that $ \displaystyle \lim_{n \to \infty} \alpha_n = \lim_{n \to \infty} \beta_n $.
	\end{lemma}
	
	\begin{proof}
		By \eqref{dependencies}, Proposition~\ref{prop_limit_b_n} and Corollary~\ref{cor_lim_a_n_lim_c_n}, we have that
		\begin{align*}
		\lim_{n \to \infty} \alpha_n p 
		&= \lim_{n \to \infty} (\alpha_{n + 1} (1 - a_n(1 - 2p) - b_n (1 - 3p) - p)) \\
		&= \lim_{n \to \infty} (\beta_{n + 1} p + p b_n +c_n (1 - 2p)) =  \lim_{n \to \infty} \beta_n p.
		\qedhere
		\end{align*}
	\end{proof}
	
	We require two further technical lemmata before we can present the proof of Proposition~\ref{limit_alpha_beta_gamma}. 
	
	\begin{lemma} \label{lemma_formula_fraction}
		For $ n \geq 2 $ we have that 
		\begin{align*}
		\frac{b_{n + 1} - c_{n + 1}}{c_{n + 1}} 
		= \frac{b_n - c_n}{c_n} \frac{c_n (1 - 2p) + b_n p}{b_n + c_n (1 - p) + \frac{b_n^2}{c_n} (1 - 2p)}.
		\end{align*}
	\end{lemma}
	
	\begin{proof}
		Using the recursive formulas given in \eqref{rec_formulas} for $ c_{n + 1} $ and $  b_{n + 1} $, we have
		\begin{align*}
		\frac{b_{n + 1} - c_{n + 1}}{c_{n + 1}}
		&= \frac{b_n c_n (2 - 3 p) + b_n^2 (1 - p) + c_n^2 p - b_n c_n - c_n^2 (1 - p) - b_n^2 (1 - 2 p)}{b_n c_n + c_n^2 (1 - p) + b_n^2 (1 - 2 p)} \\
		&= \frac{b_n c_n (1 - 3p) + b_n^2 p + c_n (2p - 1)}{b_n c_n + c_n^2 (1 - p) + b_n^2 (1 - 2 p)}
		= \frac{(b_n - c_n) (c_n (1 - 2p) + b_n p)}{b_n c_n + c_n^2 (1 - p) + b_n^2 (1 - 2 p)}.
		\qedhere
		\end{align*}
	\end{proof}
	
	\begin{lemma} \label{lemma_bound_1_p}
		For $ n \geq 2 $ it holds that 
		\begin{align*}
		\frac{c_n (1 - 2p) + b_n p}{b_n + c_n (1 - p) + \frac{b_n^2}{c_n} (1 - 2p)} \leq 1 - p.
		\end{align*}
	\end{lemma}
	
	\begin{proof}
		The result follows as $c_n (1 - 2p) + b_n p \leq (1 - p) (b_n + c_n (1 - p) + b_n^2 (1 - 2p)/c_n )$ which is equivalent to $ 0 \leq b_n c_n (1 - 2p) + c_n^2 p^2 + b_n^2 (1 - 2p) (1 - p) $.
	\end{proof}
	
	\begin{corollary} \label{cor_lim_b_n_over_c_n}
		$ \displaystyle\lim_{n \to \infty} b_n / c_n = 1 $
	\end{corollary}
	
	\begin{proof}
		The result is a consequence of the following observation.  For $ n \geq 2 $, by Lemmata \ref{lemma_formula_fraction} and \ref{lemma_bound_1_p},
		\[
		\left\lvert \frac{b_{n + 1}}{c_{n+ 1}} - 1 \right\rvert = \frac{b_{n + 1} - c_{n + 1}}{c_{n + 1}}
		\leq \frac{b_n -  c_n}{c_n} (1 - p) \leq \frac{b_2 -  c_2}{c_2} (1 - p)^{n - 1}. 
		\qedhere
		\]
	\end{proof}
	
	\begin{proposition} \label{limit_alpha_beta_gamma}
		We have that $\displaystyle \lim_{n \to \infty} \alpha_{n} = \lim_{n \to \infty} \beta_{n} = 2/5$ and $\displaystyle \lim_{n \to \infty} \gamma_{n} = 1/5$.
	\end{proposition}
	
	\begin{proof}
		Define $ a_{n + 1}' \coloneqq c_n (2 - p) p + c_n^2 (1 - 3 p) + b_n p (3 - 4 p) + b_n c_n (2 - 9 p + 9 p^2) $, which is, due to \eqref{rec_formulas}, the numerator of $ a_{n + 1} $. This in tandem with \eqref{rec_formulas} implies that 
		\begin{align*}
		\beta_{n + 1} 
		&= \frac{(b_n + c_n)(1 - p)p}{d_n} \: \frac{a_{n + 1}'}{a_{n + 1}'}\\
		&= \frac{(b_n + c_n)(1 - p)p}{c_n (2 - p) p + c_n^2 (1 - 3 p) + b_n p (3 - 4 p) + b_n c_n (2 - 9 p + 9 p^2)} \frac{a_{n + 1}'}{d_n} \\
		&= \frac{(b_n + c_n)(1 - p)p}{2 c_n (1 - p) p}  \frac{2 c_n (1 - p)\,p\,a_{n + 1}}{c_n (2 - p) p + c_n^2 (1 - 3 p) + b_n p (3 - 4 p) + b_n c_n (2 - 9 p + 9 p^2)} \\
		&= \left ( \frac{1}{2} \frac{b_n}{c_n} + \frac{1}{2} \right )
		\frac{2 (1 - p)\,p\,a_{n + 1}}{(2 - p) p + c_n (1 - 3 p) + (b_n/c_n) p (3 - 4 p) + b_n (2 - 9 p + 9 p^2)}.
		\end{align*}
		It follows by the Corollaries~\ref{cor_lim_a_n_lim_c_n} and~\ref{cor_lim_b_n_over_c_n} that 
		\[ \beta \coloneqq \lim_{n \to \infty} \beta_{n + 1} = \lim_{n \to \infty} \frac{2 (1 - p)}{(2 - p) + (3 - 4p)} 
		= \frac{2}{5}.  \]
		Lemma~\ref{gleicher_GW_beta_alpha} yields $\displaystyle \lim_{n \to \infty} \alpha_n = 2/5 $, and so, $\displaystyle 1/5 = \lim_{n \to \infty} 1 - \alpha_n - \beta_n = \lim_{n \to \infty} \gamma_n$.
	\end{proof}

\section*{Acknowledgements}
	
The authors acknowledge the support of the Deutsche Forschungsgemeinschaft (DFG grant Ke 1440/3-1). Part of this work was completed while the authors were visiting the Mittag-Leffler institute as part of the research program \textsl{Fractal Geometry and Dynamics}. We are extremely grateful to the organisers and staff for their very kind hospitality, financial support and a stimulating atmosphere. Additionally, K.\,Sender acknowledges support from the DAAD PROMOS program. 

\bibliographystyle{plain}
\bibliography{bib}
	
\end{document}